\documentclass[jmp,preprint]{revtex4-1}

\usepackage{graphicx}
\usepackage{dcolumn}
\usepackage{bm}

\usepackage[utf8]{inputenc}
\usepackage[T1]{fontenc}
\usepackage{mathptmx}
\usepackage{etoolbox}
\usepackage{subcaption,bbm,amsmath,amsthm,amsfonts,amssymb,babel}

\newtheorem{theorem}{Theorem}[section]
\newtheorem*{theorem*}{Theorem}
\newtheorem{corollary}[theorem]{Corollary}
\newtheorem{conjecture}[theorem]{Conjecture}
\newtheorem{lemma}[theorem]{Lemma}
\newtheorem{proposition}[theorem]{Proposition}
{\theoremstyle{definition}
\newtheorem{definition}[theorem]{Definition}

\newtheorem*{remark*}{Remark}}

\newcommand{\eps}{\varepsilon}
\newcommand{\CC}{\mathbb{C}}
\newcommand{\RR}{\mathbb{R}}
\newcommand{\QQ}{\mathbb{Q}}
\newcommand{\ZZ}{\mathbb{Z}}
\newcommand{\PP}{\mathbb{P}}

\newcommand{\op}[1]{\operatorname{#1}}
\newcommand{\GL}{\op{GL}}
\newcommand{\PGL}{\op{PGL}}
\newcommand{\SL}{\op{SL}}
\newcommand{\SO}{\op{SO}}
\newcommand{\SU}{\op{SU}}
\newcommand{\U}{\op{U}}
\newcommand{\tr}{\op{tr}}
\renewcommand{\Im}{\op{Im}}
\renewcommand{\Re}{\op{Re}}

\newcommand{\ovl}[1]{\overline{#1}}
\usepackage{color}
\usepackage{dsfont}

\makeatletter
\def\@email#1#2{%
 \endgroup
 \patchcmd{\titleblock@produce}
  {\frontmatter@RRAPformat}
  {\frontmatter@RRAPformat{\produce@RRAP{*#1\href{mailto:#2}{#2}}}\frontmatter@RRAPformat}
  {}{}
}%
\makeatother
\begin{document}

\preprint{}

\title[On Unitary Monodromy of Second-Order Ordinary Differential Equations]{On Unitary Monodromy of Second-Order Ordinary Differential Equations}

\author{D. Darrow}
\affiliation{ 
Department of Mathematics, MIT, Cambridge, MA 02139, USA}%
\email{ddarrow@mit.edu.}

\author{E. Chen}%
\affiliation{ 
Stanford University, Stanford, CA 94305, USA}%

\author{A. Zitzewitz}
\affiliation{%
MIT, Cambridge, MA 02139, USA}%

\date{\today}

\begin{abstract}
Given a second-order, holomorphic, linear differential equation $Lf=0$ on a punctured Riemann surface, we say that its monodromy group $G\subset\GL(2,\CC)$ is \emph{unitary} if it preserves a non-degenerate Hermitian form $H$ on $\CC^2$ under the action $g\circ H=g^\dagger H g$. In the present work, we give two sets of necessary and sufficient conditions for a monodromy group $G\subset\GL(2,\CC)$ to be unitary. First, in the case that the natural representation of $G$ on $\CC^2$ is irreducible, we show that unitarity is equivalent to a set of easily-verified trace conditions on local monodromy matrices; in the case that the representation is reducible, we show that $G$ is unitary if and only if it is contained in one of two model subgroups of $\GL(2,\CC)$. Second, we show that unitarity is equivalent to a criterion on the real dimension of the algebra $A$ generated by a rescaled group $G'\subset\SL(2,\CC)$: that $\dim(A)=1$ if $G\subset S^1$ is scalar, $\dim(A)=2$ if $G$ is abelian, $\dim(A)=3$ if $G$ is non-abelian but its action on $\CC^2$ is reducible, and $\dim(A)=4$ otherwise. Our results directly extend the recent work of Adachi (2022, 2024), which treated Fuchsian equations with irreducible monodromy representation on the punctured Riemann sphere. We leverage these results to provide evidence that the spectrum of any Darboux operator should belong to a perturbed, squared lattice in the plane, extending a conjecture of Frits Beukers (2007). Our work makes progress towards characterizing the spectra of second-order operators on Riemann surfaces, and in particular, towards answering the \emph{accessory parameter problem} for Darboux equations.
\end{abstract}

\maketitle

\section{Introduction} \label{sec:Intro}

Consider the equation $Lf = 0$, where $L$ is a second-order, holomorphic differential operator on a (punctured) Riemann surface. 
We say the monodromy group $G\subset\GL(2,\CC)$ of this equation is \emph{unitary} if it preserves a non-degenerate Hermitian form $H$ under the action \(g:H \mapsto g^{\dagger} H g\). We consider the following problem:
\begin{equation}
	\tag{A}\label{qs:app}
	\parbox{\dimexpr\linewidth-4em}{%
		\strut\itshape
		For which operators $L$ does the equation $Lf = 0$ have a unitary monodromy group?%
		\strut
	}
\end{equation}

In the present work, we give two sets of necessary and sufficient conditions for a general, second-order operator $L$ to have a unitary monodromy group, we construct the Hermitian form $H$ in both cases, and we apply this result to better understand the point-spectra of differential operators. As a key example, we consider the \emph{Heun} equation~\cite{heun1888theorie}, which encompasses any second-order equation on the Riemann sphere $\PP^1$ with four regular singularities. The Heun equation reads
\begin{equation} \label{eq:Heun}
\frac{d^2y}{dx^2} + \left(\frac{\gamma}{x} + \frac{\delta}{x-1} + \frac{\varepsilon}{x-a}\right) \frac{dy}{dx} + \frac{\alpha \beta x - \frac{B}{4}}{x(x-1)(x-a)} y = 0,
\end{equation}
where $B\in\CC$, and where the parameters $\gamma,\delta,\varepsilon,\alpha,\beta\in \RR$ satisfy $\gamma + \delta + \varepsilon = 1 + \alpha + \beta$. In this context,~\eqref{qs:app} can be reformulated as the \emph{accessory parameter problem}:
\begin{equation}
	\tag{A'}\label{qs:app2}
	\parbox{\dimexpr\linewidth-4em}{%
		\strut\itshape
		For which values $B\in\CC$ does~\eqref{eq:Heun} have a unitary monodromy group?%
		\strut
	}
\end{equation}
Even restricted to the Heun equation, this problem has several important applications. For one, the analytic Langlands correspondence claims that eigenfunctions of \emph{Hecke} (or \emph{Baxter}) and \emph{Gaudin} operators on the moduli space of parabolic $\PGL(2,\CC)$ bundles on $\PP^1$ are encoded by rank-two Fuchsian equations with real-valued (and thus unitary) monodromy~\cite[Sec.~3.3]{Etingof:2023drx}. The simplest nontrivial example of the analytic Langlands correspondence occurs when $N=4$, which leads to the accessory parameter problem for the Heun equation~\cite{alma9949400487002959,Etingof:2023drx}. In a different direction, the accessory parameter problem for the Heun equation can be shown to be equivalent to the problem of classifying spherical quadrilaterals~\cite{Eremenko2016rect}, where each accessory parameter satisfying \eqref{qs:app2} corresponds to a unique spherical quadrilateral with ($\pi$-scaled) angles $(1-\gamma,1-\delta,1-\eps,\alpha-\beta)$. Eremenko and Gabrielov used this correspondence to classify spherical rectangles, for which each angle has a half-integer value~\cite{Eremenko2016rect}.

Monodromy and the accessory parameter have also garnered wide interest in different contexts. Historically, monodromy of the Heun equation was studied in the context of isomonodromic deformations, tracing back to Paul Painlev\'{e} in the early 20$^\text{th}$ century. Specifically, the modern isomonodromic framework shows that Painlev\'{e}'s sixth equation characterizes the isomonodromic deformations of rank-two Fuchsian systems with four regular singular points~\cite{Takemura2005HP}; the scalar Heun equation~\eqref{eq:Heun} can be written in this form by lifting it to a first-order equation for $Y = (y,\frac{dy}{dx})$. The accessory parameter has also seen significant attention---for instance, Nehari~\cite{nehari1949accessory} examined the impact of the accessory parameter on the Schwarz map generated by the ratio of two independent solutions to a Fuchsian equation. Additionally, Keen, Rauch, and Vasquez~\cite{keen1979moduli} have studied the accessory parameter as a parameter relating different covering maps of the punctured torus, and Beukers~\cite{beukers2002dwork} has studied the accessory parameter's impact on the $p$-adic radius of convergence for solutions to these differential equations.

A partial answer to~\eqref{qs:app2} follows from classical work of Fricke and Klein~\cite[p.~365--366]{Fricke1897}. Namely, if $\gamma,\delta,\eps,\alpha,\beta\in\QQ$, then it is well-known that the monodromy group is determined by the so-called \emph{Fricke parameters}; if $P,Q,R\in \SL(2,\CC)$ are the generators of the monodromy group, the Fricke parameters are as follows:
\begin{equation}\label{eq:fricke}
    \tr(P), \tr(Q), \tr(R), \tr(PQ), \tr(QR), \tr(PR), \tr(PQR).
\end{equation}
If $\gamma,\delta,\eps,\alpha,\beta\in\QQ$, the seven values \eqref{eq:fricke} are known to be real, and the monodromy representation is irreducible, then one can deduce that the monodromy group itself is unitary~\cite[Sec.~4]{reiter2011halphenstransformmiddleconvolution}. It was also shown by Morgan and Shalen~\cite[Prop.~III.1.1]{Morgan1984} and by Goldman~\cite[Thm.~4.3]{Goldman1988} that a general subgroup $G\subset \SL(2,\CC)$ is unitary if and only if $\tr (g)\in\RR$ for all $g\in G$, or equivalently, if the natural representation of $G$ on $\CC^2$ has real character. Beukers offered a stronger result in the special case of the \emph{Lamé} equation~\cite{beukers2007unitary}, i.e., the Heun equation with $\gamma=\delta=\eps=1/2$. Beukers showed that, if $\gamma=\delta=\eps=1/2$ and $\alpha=\beta=1/4$, a unitary monodromy group is equivalent to a condition on only five of the seven Fricke parameters~\eqref{eq:fricke}, and he leveraged this result to estimate the eigenvalues of the Lam\'e operator. Finally, Adachi~\cite{adachi_monodromy_2022,adachi2023unitarymonodromiesrankfuchsian} has recently proved necessary and sufficient trace conditions for a generic rank-two Fuchsian equation on the punctured Riemann sphere to carry a unitary monodromy group, by studying the moduli space of irreducible monodromy groups as an algebraic variety. Adachi's work provides important geometric insight, but is also highly technical, and does not cover the full range of second-order differential equations. Specifically, it applies only to equations with regular singularities on the punctured Riemann sphere, and only in the generic case that the monodromy representation is irreducible. 


In the present work, we work towards answering \eqref{qs:app} in the case of general second-order differential operators, by characterizing the unitary subgroups $G$ of $\GL(2,\CC)$. We provide two sets of equivalent criteria for unitarity. First, in the case that the natural representation of $G$ on $\CC^2$ is \emph{irreducible}, we relate unitarity to a set of easily-verified trace conditions on monodromy matrices, directly extending previous results in this direction~\cite{Morgan1984,Goldman1988,beukers2007unitary,adachi_monodromy_2022,adachi2023unitarymonodromiesrankfuchsian}. For Fuchsian operators on the Riemann sphere, such as the Heun equation \eqref{eq:Heun}, our proof yields an elementary alternative to the recent work of Adachi~\cite{adachi_monodromy_2022,adachi2023unitarymonodromiesrankfuchsian}. Indeed, our result relies only on linear algebra at the level of monodromy generators, and is thus highly accessible to non-specialists and well-adapted to numerical experimentation (cf.~Section~\ref{sec:asymp}). In the case where the representation is reducible, we show that $G$ is unitary if and only if it is conjugate to a subset of one of two model subgroups of $\GL(2,\CC)$. Second, we show that both reducible and irreducible cases can be understood in terms of a new algebraic criterion. Namely, unitarity is equivalent to a criterion on the real dimension of the algebra $A$ generated by a rescaled group $G'\subset\SL(2,\CC)$: that $\dim(A)=1$ if $G\subset S^1$ is trivial, $\dim(A)=2$ if $G$ is abelian, $\dim(A)=3$ if $G$ is non-abelian but its representation on $\CC^2$ is reducible, and $\dim(A)=4$ otherwise. 

In all cases, we are able to explicitly construct the Hermitian form $H$ preserved by $G$. We will see that a constructive proof is necessary to recover our second criterion for unitarity, in terms of $\dim(A)$, and is instrumental in moving beyond the case of Heun operators. Connecting with another result of Beukers~\cite[Prop.~4]{beukers2007unitary}, quoted as Proposition~\ref{prop:prop3} below, this also allows explicit construction of real-analytic solutions to second-order differential equations that admit them.

Finally, we apply our results to extend a conjecture of Beukers on the spectra of Lam\'e operators~\cite[Conj.~1]{beukers2007unitary}. Leveraging a similar asymptotic estimate and numerical algorithm as Beukers, we provide evidence that the spectra of Darboux operators (which differ from Heun operators only by a change of variable~\cite{sirota2006heun,10.1063/1.3367079}) should closely align with a (squared) parameter-dependent lattice. In so doing, we also provide evidence that our results can be strengthened in the case of Heun and Darboux equations, with fewer trace conditions required to determine unitarity.

\section{Background on Unitary Monodromy} 
We give a brief background on the central topics necessary for our analysis. In particular, we briefly introduce the fundamentals of monodromy groups and unitarity, and we outline how they apply to the Heun and Darboux equations.

\label{sec:Background}
\subsection{Monodromy and Local Exponents}
Fix a Riemann surface $X$ and a discrete set of punctures $\{x_k\}=S\subset X$, and consider the differential equation 
\begin{equation}\label{eq:example}
    \frac{d^2y}{dx^2} + P(x) \frac{dy}{dx} + Q(x) y = 0,
\end{equation}
where $P(x)$ and $Q(x)$ are holomorphic on $X\setminus S$. We remark that $X$ is general, and that the equation need not be Fuchsian---i.e., it is not necessary that $xP(x)$ and $x^2Q(x)$ be holomorphic. Let $\Gamma_i$ be closed, simple, counter-clockwise loops around each $x_i$ that do not encircle any $x_j\neq x_i$, each starting at a common point $p \notin S$.

Let $y_1$ and $y_2$ be two linearly independent solutions of~\eqref{eq:example} in a neighborhood $U\ni p$. We can analytically continue $y_1$ and $y_2$ about $\Gamma_k$, resulting in the new solutions $\widetilde{y}_1$ and $\widetilde{y}_2$ in $U$, respectively. In such a case, there exists an  $M_k\in\GL(2,\CC)$ such that
\[M_k \begin{pmatrix}
    \,y_1\\
    \,y_2
\end{pmatrix} =\begin{pmatrix}
    \,\widetilde{y}_1\\
    \,\widetilde{y}_2
\end{pmatrix}\]
everywhere in $U$. The matrix $M_k$ is known as the \emph{monodromy matrix} around $\Gamma_k$, and the group $G\subset\GL(2,\CC)$ generated by $\{M_k\}$ is known as the \emph{monodromy group} of~\eqref{eq:example}. Notably, the mapping 
\[\pi_1(X\setminus S)\to G,\;\Gamma_k\mapsto M_k\]
is a group homomorphism.


A key application of monodromy matrices is that they quantify the asymptotic behavior of a differential equation near its singularities. In a sufficiently small neighborhood of a regular singularity $x_k$, we can generically find $y_1$ and $y_2$ such that $y_1 = (x-x_k)^{a_k} u_1$ and $y_2 = (x-x_k)^{b_k} u_2$, where $u_1$ and $u_2$ are nonzero analytic functions and $a_k,b_k\in\CC$. We refer to $a_k$ and $b_k$ as the \textit{local exponents} of~\eqref{eq:example} at $x_k$. In general, if the monodromy matrix $M_k$ has eigenvalues $(e^{2\pi i a_k},e^{2\pi i b_k})$, we define the local exponents at $x_k$ to be $(a_k,b_k)$, modulo $\ZZ$.


\subsection{Unitarity} \label{subsec:unitarity}
Given a nondegenerate Hermitian form $H$ on $\CC^2$, we define
the \emph{unitary group} of $H$ to be
\[\U(H) = \{g \in \GL(2,\CC) \;|\; g^\dagger H g = H\},\]
generalizing the classical unitary group $\U_2 = \U(\mathbbm{1}_2)$. In general, we define \emph{unitarity} as follows:
\begin{definition}\label{def:unitary}
    A matrix group $G\subset\GL(2,\CC)$ is \emph{unitary} if $G\subset \U(H)$ for a nondegenerate Hermitian form $H$.
\end{definition}
Importantly, unitarity is preserved under conjugation; if $h \in \GL(2,\CC)$, then we have
\[h^{-1} \U(H) h = \U(h^\dagger H h).\]
With this property in mind, we can break unitarity into two basic cases. If the form $H$ is (positive or negative) definite, it is conjugate to the Euclidean form $\pm\mathbbm{1}_2$, and any group that preserves it must be conjugate to a subgroup of $\U_2$. Similarly, if the form $H$ is indefinite, the unitary group $U(H)$ is given (up to conjugation) by the following well-known result (cf.~\cite{beukers2007unitary}):
\begin{proposition}
\label{prop:prop1} 
Let 
\begin{equation}\label{eq:standard_form}
    H_0 = \begin{pmatrix}
 & +i\,\\
-i & \\
\end{pmatrix}.
\end{equation}
Then $\U(H_0)\simeq S^1\times \SL(2,\RR)$ is the group generated by $\SL(2,\RR)$ and the scalars $\lambda\in S^1$. 
\end{proposition}
\begin{remark*}
    Now, the standard embedding of $S^1\subset\GL(2,\CC)$ in terms of diagonal matrices intersects $\SL(2,\RR)$ at the two points $\{\pm I\}$; we clarify that the factor of $S^1$ in $\U(H_0)\simeq S^1\times\SL(2,\RR)$ instead parameterizes the \emph{determinant} of a matrix $g\in\U(H_0)$, so that $\{\pm I\}$ is in the kernel of the projection $\U(H_0)\to S^1$.
\end{remark*}

\subsection{The Lam\'e, Heun, and Darboux Equations}\label{sec:Background_eqs}
Several aspects of our analysis are inspired by Frits Beukers' work on the Lamé equation~\cite{beukers2007unitary}. In algebraic form, the Lam\'e equation reads as follows:
\begin{equation} \label{eq:lameAlgebraic}
p(x)\frac{d^2y}{dx^2} + \frac{1}{2}p'(x)\frac{dy}{dx} + (Ax - B) y = 0,
\end{equation}
where $p(x)= 4x(x-1)(x-a)$ and $A,B\in\CC$, and $x$ takes values in $\PP^1$. Comparing to~\eqref{eq:Heun}, we see that the Lamé equation is a special case of the Heun equation with $\gamma=\delta=\eps=1/2$; Beukers was interested in the particular subcase $A=1/4$, or equivalently, $\alpha=\beta=1/4$.

As with the Heun equation, the Lamé equation has four regular singularities: $z_1=0$, $z_2=1$, $z_3=a$, $z_0=\infty$. Unlike the Heun equation, however, the first three singularities all carry the same local exponents $(0,1/2)$. From the discussion above, this implies that the corresponding monodromy matrices $P=M_1$, $Q=M_2$, and $R=M_3$ have the eigenvalues $\pm 1$; we refer to such matrices as \emph{reflections} below, though they may not be orthogonal reflections. Moreover, the product $M = PQR$ (which has local exponents $(1/4,1/4)$) is \emph{parabolic} in the following sense:
\begin{definition}\label{def:parabolic}
    A matrix $M$ is \emph{parabolic} if it has two identical eigenvalues in $S^1$ and is not diagonalizable.
\end{definition}
Using these properties of the monodromy group, Beukers was able to give a necessary and sufficient criterion for unitarity in the case of the Lamé equation with $A=1/4$; the following proposition is transcribed from~\cite[Prop.~2]{beukers2007unitary}:
\begin{proposition}[Beukers, 2007]\label{prop:beukers}
    Suppose $P,Q,R\in\GL(2,\CC)$ are reflections, and their product $M=PQR$ is parabolic. Let $G$ be the group generated by $P$, $Q$, and  $R$. Then the following statements are equivalent:
    \begin{enumerate}
        \item The group $G$ is unitary.
        \item The traces $\tr(PQ)$, $\tr(QR)$, and $\tr(PR)$ are real.
        \item The traces $t_{PQ}=\tr(PQ)$ and $t_{QR}=\tr(QR)$ are real and satisfy $(t_{PQ}^2-4)(t_{QR}^2-4)\geq 16$.
    \end{enumerate}
\end{proposition}
This result provides significant insight into the accessory parameter problem~\eqref{qs:app2} in Beukers' case of interest. An accessory parameter $B$ gives rise to unitary monodromy if and only if the local monodromy matrices satisfy either of the above trace conditions. These trace conditions are easily verified on a computer, so Beukers was able to leverage his result to numerically estimate accessory parameters that give rise to unitary monodromy.

Recall that the Heun equation~\eqref{eq:Heun} encompasses all second-order equations on $\PP^1$ with at most four singular points, greatly generalizing the setting of Proposition~\ref{prop:beukers}. With this level of generality, however, one sacrifices the strong constraints on local monodromy matrices that formed the basis of Beukers' argument; instead, the local exponents of the Heun equation are given by (cf.~\cite{xia2021isomonodromy})
\begin{equation}\label{eq:localexp_heun}
    \begin{aligned}
    &\text{At } z=0 ,\quad \left(0, 1-\gamma \right)\\
    &\text{At } z=1 ,\quad \left(0, 1-\delta \right)\\
    &\text{At } z=a ,\quad \left(0, 1-\varepsilon \right)\\
    &\text{At } z=\infty ,\quad \left(\alpha, \beta \right).
    \end{aligned}
\end{equation}
Although one might be interested in the general case $\gamma,\delta,\eps,\alpha,\beta\notin\RR$, the expressions \eqref{eq:localexp_heun} for the local exponents severely restrict the parameter space that could exhibit unitary monodromy:
\begin{proposition}\label{prop:onlyreal}
    The monodromy group of the Heun equation~\eqref{eq:Heun} can only be unitary if $\gamma,\delta,\varepsilon\in\RR$ and either $\alpha,\beta\in\RR$ or $\alpha-\ovl{\beta}\in\ZZ$. 
\end{proposition}
\begin{proof}
    Write $P$, $Q$, and $R$ for the monodromy matrices about $0$, $1$, and $a$, respectively, and suppose $G$ preserves a Hermitian form $H$. From the expressions~\eqref{eq:localexp_heun}, we find $\det(P)=e^{-2\pi i\gamma}$, and thus
    \[\det(P^\dagger H P) = |e^{-2\pi i\gamma}|^2\det(H) = e^{4\pi\Im(\gamma)}\det(H).\]
    Since $H$ must be preserved under the action of $G$, this implies that $\gamma\in\RR$, and similarly for $\delta$, $\varepsilon$, and $\alpha+\beta$. It will follow from Lemma \ref{lem:realtraces} below that we must also have 
    \[2\cos(\pi(\alpha - \beta)) = \frac{\tr(PQR)}{\sqrt{\det(PQR)}}\in\RR,\]
    which can only be satisfied if either $\alpha-\beta\in\RR$ or $\Re(\alpha-\beta)\in\ZZ$. These correspond to the two conditions given in the proposition statement.
\end{proof}


Finally, we note that the Heun equation is commonly expressed in terms of elliptic functions, using a change of variable discussed in~\cite{sirota2006heun,10.1063/1.3367079}. For this, take values $\omega_1,\omega_2\in\CC$ (which we will fix shortly), and let $\wp(z)$ be the Weierstrass function~\cite{cartan1995elementary} with half-periods $\omega_1$, $\omega_2$, and $\omega_3 = \omega_1+\omega_2$. Write $e_i = \wp(\omega_i)$, and fix $\omega_i$ such that
\[a = \frac{e_2-e_1}{e_3-e_1}.\]
Substitute $\wp(z) = e_1 + (e_2-e_1)x$, and let 
\begin{equation}\label{eq:params_D2H}
    m_0 = \alpha - \beta - \frac{1}{2}, \qquad m_1 = \frac{1}{2}-\gamma,\qquad
m_2 = \frac{1}{2}-\delta, \qquad m_3 = \frac{1}{2}-\varepsilon.
\end{equation}
Finally, write
\[u = x^{-m_1/2}(x-1)^{-m_2/2}(x-a)^{-m_3/2}y.\]
With this change of variables, we convert~\eqref{eq:Heun} into the \emph{Darboux} equation:
\begin{equation} \label{eq:Darboux}
    \frac{d^2u}{dz^2} - \left(\sum_{i=0}^3 m_i (m_i+1) \wp(z-\omega_i)\right)u = B'u,
\end{equation}
where we write $\omega_0 = 0$.
Here, $B' = (e_3 - e_1)B + B_0$ for a fixed $B_0 = B_0(\alpha,\beta,\gamma,\delta,\eps)$. A full expression of $B_0$ is given in~\cite{2001math.....12179T}, for instance, but we do not make use of it here. The Darboux equation has the local exponents
\begin{equation}\label{eq:localexp_darboux}
    (-m_i,1+m_i)\qquad\text{at } z=\omega_i,
\end{equation}
which can be deduced from~\eqref{eq:localexp_heun} by applying the described coordinate transform. Notably, all four monodromy matrices have unit determinant \emph{a priori}, so the analogue of Proposition~\ref{prop:onlyreal} would not restrict us to particular parameter values; indeed, we will see in Section~\eqref{sec:HeunAnalysis} that the Darboux equation satisfies~\eqref{qs:app} for a far wider range of parameters than the Heun equation.

The Lamé equation~\eqref{eq:lameAlgebraic} takes on a particularly simple form in these coordinates, as $m_1$, $m_2$, and $m_3$ all vanish:
\begin{equation} \label{eq:Darboux_Lame}
    \frac{d^2u}{dz^2} - m_0(m_0+1)\wp(z) u = B'u,
\end{equation}
noting that $u(z)=y(x)$ in this case.

\section{Unitary Matrix Groups}\label{sec:matrices}
We can make progress on the problem~\eqref{qs:app} by considering monodromy groups in purely linear-algebraic terms, i.e., as subgroups of $\GL(2,\CC)$. We emphasize that the correspondence between a linear differential operator and its monodromy group is in general highly transcendental. Nevertheless, an answer to the following question provides a group-theoretic criterion for any second-order operator to satisfy~\eqref{qs:app}, once the group $G$ has been determined:
\begin{equation}
	\tag{B}\label{qs:app3}
	\parbox{\dimexpr\linewidth-4em}{%
		\strut\itshape
		Which subgroups $G\subset\GL(2,\CC)$ are unitary, in the sense of Definition~\ref{def:unitary}?
		\strut
	}
\end{equation}
We answer this question completely in the present section. Proposition~\ref{prop:general} first answers the question for groups $G\subset\SL(2,\CC)$ generated by three elements, as we encounter in the Heun~\eqref{eq:Heun} and Darboux~\eqref{eq:Darboux} equations; this case also follows from the results of Morgan and Shalen~\cite[Prop.~III.1.1]{Morgan1984} and Goldman~\cite[Thm.~4.3]{Goldman1988}, but the constructive proof we offer is necessary to extend to arbitrary subgroups of $\GL(2,\CC)$ in Theorem~\ref{thm:verygeneral}. The conclusion of Theorem~\ref{thm:verygeneral} overlaps substantially with a recent result of Adachi~\cite{adachi_monodromy_2022,adachi2023unitarymonodromiesrankfuchsian} for Fuchsian equations on the Riemann sphere, but offers several distinct advantages: it applies to irregular equations, multiply-connected Riemann surfaces, and reducible monodromy; it is comparatively elementary and accessible to non-specialists; and it offers a new viewpoint on unitarity based on the dimension of the real algebra generated by a rescaled monodromy group.

For any subgroup $G\subset\GL(2,\CC)$, we define a rescaled group
\begin{equation*}
    G' = \left\{\pm g/\sqrt{\det g}\;\big|\; g\in G\right\} \subset\SL(2,\CC).
\end{equation*}
In one direction, we show that for subgroups $G\subset\GL(2,\CC)$ for which the natural representation of $G$ on $\CC^2$ is \emph{reducible}, the rescaled group $G'$ must be contained within either $\SO(2)$ or the group $\RR\rtimes\RR^\times\subset\SL(2,\RR)$ of real, upper-triangular matrices; if the representation of $G$ on $\CC^2$ is \emph{irreducible}, we will show that unitarity is instead equivalent to a set of easily-verified trace conditions on an arbitrary set of generators. In a different direction, we show  that these same constraints can be expressed as a more natural algebraic statement about $G'$: for both reducible and irreducible cases, unitarity is equivalent to a strict criterion on the real dimension of the algebra $A$ generated by $G'$. 


One direction of our general theorem follows straightforwardly from Proposition~\ref{prop:prop1}. We note that $G$ is unitary if and only if $G'$ is unitary and $|\det g| = 1$ for all $g\in G$, so we can restrict to subgroups $G\subset\SL(2,\CC)$ without any loss of generality. 

\begin{lemma}\label{lem:realtraces}
    Suppose $G\subset\SL(2,\CC)$ is unitary, in the sense of Definition~\ref{def:unitary}. Then all elements of $G$ have real trace.
\end{lemma}
\begin{proof}
    Suppose $G\subset\U(H)$ for a nondegenerate Hermitian form $H$. If $H$ is indefinite, for one, it is conjugate to the form $H_0$ of Proposition~\ref{prop:prop1}. But then $G$ is conjugate to a subgroup of $\U(H_0)\simeq \U_1\otimes\SL(2,\RR)$; since $G\subset\SL(2,\CC)$ by hypothesis, it must in fact be conjugate to a subgroup of $\SL(2,\RR)$, from which the claim follows.

    If $H$ is definite, it is conjugate to the Euclidean form (up to sign), and $G$ is conjugate to a subgroup of $\SU_2$. Any element of $\SU_2$ has eigenvalues $\lambda,\lambda^{-1}\in S^1$, which implies our result.
\end{proof}

The converse of Lemma~\ref{lem:realtraces} does not generally hold: indeed, consider the matrices
\begin{equation}\label{eq:reducible_fail}
    P=\begin{pmatrix} 
    \,1\phantom{.} & z_p \\  & 1
\end{pmatrix},\qquad Q=\begin{pmatrix} 
    \,1\phantom{.} & z_q \\  & 1
\end{pmatrix}, \qquad z_p\ovl{z}_q\notin\RR.
\end{equation}
It is clear that $P$, $Q$, and all products thereof have real trace. However, suppose they both preserve a Hermitian form
\[H=\begin{pmatrix} \,h_{11} & h_{12} \\ \,\ovl{h}_{12} & h_{22} \end{pmatrix},\]
so that $P^\dagger H = H P^{-1}$ and $Q^\dagger H = H Q^{-1}$. Matching elements shows that $h_{11}=0$, implying $h_{12}\neq 0$, as well as $\ovl{z}_ph_{12} = -z_p \ovl{h}_{12}$ and $\ovl{z}_qh_{12} = -z_q \ovl{h}_{12}$. These conditions can only be met if $z_p\ovl{z}_q\in\RR$, and we reach a contradiction.

We first demonstrate that reducible monodromy groups (like \eqref{eq:reducible_fail}) are the only context in which the converse of Lemma~\ref{lem:realtraces} can fail. As discussed above, the next two results---Lemma~\ref{lem:gt2} and Proposition~\ref{prop:general}---follow from the work of Morgan and Shalen~\cite[Prop.~III.1.1]{Morgan1984} and Goldman~\cite[Thm.~4.3]{Goldman1988}. Even still, in order to generalize to Theorem~\ref{thm:verygeneral} below, we require a constructive proof of these statements; that is, given matrices $P,Q,R\in\SL(2,\CC)$, we must identify the particular Hermitian form $H$ they preserve. 

\begin{lemma} \label{lem:gt2}
    Fix matrices $P,Q,R\in\SL(2,\CC)$ with $\tr(P)\neq \pm2$, and let $G$ be the group generated by $\{P,Q,R\}$. Suppose the natural representation of $G$ on $\CC^2$ is irreducible. Then $G$ is unitary if and only if
    \begin{equation}\label{eq:allreal}
        \tr (P),\tr (Q), \tr (R), \tr (PQ), \tr (QR), \tr (PR), \tr (PQR)\in\RR.
    \end{equation}
\end{lemma}
\begin{proof} 
The forward direction (unitarity$\implies$real traces) follows from Lemma~\ref{lem:realtraces}, so we focus here on the converse direction.

Suppose that~\eqref{eq:allreal} holds. Since $\det(P)=1$ and $\tr(P)\neq\pm 2$, the matrix $P$ is necessarily diagonalizable. By conjugating our matrices appropriately, we can put $P$ in the form
\begin{equation}\label{eq:eigenbasis}
    P = \begin{pmatrix}\,
\lambda_P & \\
 & \lambda_P^{-1}
\end{pmatrix},\qquad \lambda_P\neq \pm 1,
\end{equation}
and in the same basis, we write
\begin{equation}\label{eq:QR_elements}
    Q = \begin{pmatrix} \,q_{11} & q_{12} \\ \,q_{21} & q_{22} \end{pmatrix},\qquad R = \begin{pmatrix} \,r_{11} & r_{12} \\ \,r_{21} & r_{22} \end{pmatrix}.
\end{equation}
We split the remainder of our proof into cases, depending on whether $|\tr(P)|>2$ or $|\tr(P)|<2$.

\emph{First case.} Suppose that $|\tr(P)|>2$, so that $\lambda_P\in\RR$. Since we know that $\lambda_P\neq\pm 1$, the requirement that $\tr(Q),\tr(PQ)\in\RR$ yields two linearly-independent conditions on $q_{11}$ and $q_{22}$:
\[q_{11} + q_{22}\in\RR,\qquad \lambda_Pq_{11} + \lambda_P^{-1}q_{22}\in\RR.\]
Combining these appropriately shows that $q_{11},q_{22}\in\RR$, and we similarly find $r_{11},r_{22}\in\RR$. But then, as
\begin{equation}\label{eq:QR_PQR}
    \begin{gathered}
        \tr(QR)=q_{11}r_{11} + q_{22}r_{22}+q_{12}r_{21}+q_{21}r_{12}\in\RR,\\
        \tr(PQR)=\lambda_P q_{11}r_{11} + \lambda_P^{-1}q_{22}r_{22}+\lambda_Pq_{12}r_{21}+\lambda_P^{-1}q_{21}r_{12}\in\RR,
    \end{gathered}
\end{equation}
we deduce that $q_{12}r_{21},q_{21}r_{12}\in\RR$. Consider a conjugation of our group by the matrix
\begin{equation}\label{eq:transform}
    T=\begin{pmatrix}\, t^{-\frac{1}{2}} &  \\  & t^{\frac{1}{2}} \end{pmatrix},
\end{equation}
with $t\in\CC$. This transformation preserves $P$ and the diagonal elements of $Q$ and $R$, but maps $q_{12}\mapsto q_{12}t^{-1}$ and $q_{21}\mapsto q_{21}t$, and likewise for $R$; notably, it leaves $q_{12}r_{21}$ and $q_{21}r_{12}$ unchanged. 

If we first suppose that $q_{12},q_{21}\neq 0$, we can fix a value $t$ such that $q_{12}\in\RR$, and thus that $r_{21}\in\RR$. Since $\det (Q)=1\in\RR$ and both diagonal elements of $Q$ are real, we must also have $q_{21}\in\RR$ (and thus $r_{12}\in\RR$) in this basis. A similar argument holds if any of the pairs $\{r_{12},r_{21}\}$, $\{q_{21},r_{12}\}$, or $\{q_{12},r_{21}\}$ are nonzero, or if at most one of $\{q_{12},q_{21},r_{12},r_{21}\}$ is nonzero.

The final possibility is that either $q_{12},r_{12}\neq 0$ or $q_{21},r_{21}\neq 0$, but the remaining off-diagonal elements vanish. This possibility is prevented by the hypothesis that the representation of $G$ on $\CC^2$ is irreducible.

\emph{Second case.} Suppose now that $|\tr(P)|<2$, so that $\lambda_P\in S^1$. We check that $G$ preserves a diagonal Hermitian form
\[H=\begin{pmatrix} \,h_{11} & \\ & h_{22}\end{pmatrix},\]
for as-of-yet unfixed (but nonzero) values $h_{11},h_{22}\in\RR$. We expand out the required equation $Q^\dagger H \overset{?}{=} HQ^{-1}$ as follows:
\begin{equation}\label{eq:elementequation}
    \begin{pmatrix}
       \, h_{11}\ovl{q}_{11} & h_{22}\ovl{q}_{21}\\
        \,h_{11}\ovl{q}_{12} & h_{22}\ovl{q}_{22}
    \end{pmatrix} \overset{?}{=} \begin{pmatrix}
        \,\phantom{-}h_{11}q_{22}\phantom{.} & -h_{11}q_{12}\\
        \,-h_{22}q_{21}\phantom{.} & \phantom{-}h_{22}q_{11}
    \end{pmatrix},
\end{equation}
and likewise for $R$. Matching matrix elements, we see that $Q,R\in\U(H)$ if and only if
\[q_{11} = \ovl{q}_{22},\qquad r_{11} =  \ovl{r}_{22},\qquad q_{12} =-h_{22} \ovl{q}_{21}/h_{11},\qquad r_{12} =-h_{22} \ovl{r}_{21}/h_{11}.\]
From here, one can also deduce that $q_{21}r_{12}=\ovl{q_{12}r_{21}}$. 

If none of the off-diagonal elements vanish, we fix $h_{11}=1$ and 
\[h_{22}=-\frac{r_{12}}{\ovl{r}_{21}}=-\frac{q_{12}}{\ovl{q}_{21}}.\]
Indeed, this value must be real, as
\[r_{12}r_{21} = r_{11}r_{22} - 1 \in\RR.\]
Suppose instead that $q_{12},q_{21}\neq 0$, but either $r_{12}=0$ or $r_{21}=0$; the identity $q_{21}r_{12}=\ovl{q_{12}r_{21}}$ then shows that $R$ is diagonal; the choice $h_{22} = -q_{12}/\ovl{q}_{21}$ works as before. A similar argument holds in the case $r_{12},r_{21}\neq 0$, and the same identity prevents the case $q_{12},r_{21}\neq 0$, $q_{21},r_{12}=0$ or its reverse. Finally, the irreducibility hypothesis prevents either $q_{21}=r_{21}=0$ or $q_{12}=r_{12}=0$.
\end{proof}

We now extend our result to the case where $\tr(P)=\pm 2$. Since $P$ may no longer be diagonalizable, we apply somewhat different techniques to recover a preserved Hermitian form in this case. 
\begin{proposition} \label{prop:general}
    Fix matrices $P,Q,R\in\SL(2,\CC)$, and let $G$ be the group generated by $\{P,Q,R\}$. Suppose the representation of $G$ on $\CC^2$ is irreducible. Then $G$ is unitary if and only if
    \begin{equation*}
        \tr (P),\tr (Q), \tr (R),\tr (PQ), \tr (QR), \tr (PR), \tr (PQR)\in\RR.
    \end{equation*}
\end{proposition}
\begin{proof}
    If any of $M\in\{P,Q,R\}$ satisfy $\tr(M)\neq\pm 2$, we can cycle the matrices until $\tr(P)\neq\pm 2$ and subsequently apply Lemma~\ref{lem:gt2}. As such, we suppose without loss of generality that $\tr(P),\tr(Q),\tr(R)=+2$; otherwise, we replace matrices with their negations as appropriate.
    
    Suppose now that $\tr(PQ)\neq\pm 2$, and define the matrices 
    \begin{equation}\label{eq:cycleperm}
        P' = PQ,\qquad Q' = Q^{-1},\qquad R'=R.
    \end{equation}
    These matrices generate the group $G$, and they satisfy
    \[\tr(P'),\tr(Q'),\tr(R'),\tr(P'Q'),\tr(P'R'),\tr(P'Q'R')\in\RR\]
    by the conditions~\eqref{eq:allreal}, noting that $\tr(Q^{-1}) = \tr(Q)=+2$. It remains to check $\tr(Q'R')$. For this, recall that every matrix satisfies its own characteristic polynomial---in particular,
    \[Q^2 - 2Q + I = 0,\]
    or equivalently,
    \[Q + Q^{-1} = 2I.\]
    Then we find
    \[\tr(Q'R') = \tr(Q^{-1}R) = 2\tr(R) - \tr(QR)\in\RR.\]
    As such, the matrices~\eqref{eq:cycleperm} satisfy the hypotheses of Lemma~\ref{lem:gt2}, proving unitarity. A similar argument applies to $QR$, $PR$, and $PQR=P'R'$, so we can assume that
    \[\tr(PQ),\tr(QR),\tr(PR),\tr(PQR)=\pm 2.\]

    Suppose also that $P$ is parabolic, in the sense of Definition~\ref{def:parabolic}; if none of $\{P,Q,R\}$ are parabolic, the result follows trivially (but also, the irreducibility hypothesis is violated). Bring $P$ to Jordan normal form as
    \begin{equation}\label{eq:jordan}
        P = \begin{pmatrix} \,1\phantom{.} & 1 \\  & 1\end{pmatrix},
    \end{equation}
    and write $Q$ and $R$ as in~\eqref{eq:QR_elements}. 
    
    We suppose first that $\tr(PQ)=\tr(PR)=\tr(QR)=-2$ and $\tr(PQR)=+2$, so that
    \[q_{21} = \tr(PQ) - \tr(Q) = -4,\qquad r_{21}=\tr(PR) - \tr(R)=-4=q_{21},\]
    and consequently,
    \begin{align*}\tr(QR) = q_{11}r_{11}+q_{22}r_{22} + q_{12}r_{21}+q_{21}r_{12}&=q_{11}r_{11} + q_{22}r_{22} + q_{12}q_{21} + r_{12}r_{21}\\
    &=q_{11}r_{11} + q_{22}r_{22} + q_{11}q_{22} + r_{11}r_{22} - 2\\
    &=(q_{11}+r_{22})(q_{22}+r_{11})- 2,
    \end{align*}
    applying $\det(Q)=\det(R)=1$. Then $\tr(QR)=-2$ implies either $q_{11}+r_{22}=0$ or $q_{22}+r_{11}=0$, but we also see that
    \[\tr(PQR) - \tr(QR) = q_{22}r_{21}+q_{21}r_{11} = -4(q_{22}+r_{11}) = 4,\]
    so that $q_{22}+r_{11}=-1$. From above, then, $q_{11}+r_{22}=0$, so we reach a contradiction:
    \[4 = \tr(Q)+\tr(R) = q_{11}+q_{22}+r_{11}+r_{22}=-1.\]
    As such, we see that either $\tr(PQR)=-2$ or at least one of $\tr(PQ)$, $\tr(PR)$, or $\tr(QR)$ is $+2$. In the former case, we make the replacement 
    \[P''=-PQ,\qquad Q''=Q^{-1},\qquad R''=R.\] 
    These matrices satisfy $\tr(P'')=\tr(Q'')=\tr(R'')=+2$, by hypothesis, and 
    \[\tr(P''Q'') = -\tr(PQR) = +2.\]
    In either case, then, we can suppose that 
    \[\tr(P)=\tr(Q)=\tr(R)=\tr(PQ)=+2.\]
    The same analysis as above shows that $q_{21}=0$ in this basis, and thus (from the fact that $\det(Q)=1$) that $q_{11}=q_{22}=1$. Since $\tr(PR)=+2$ would similarly imply that $r_{21}=0$, and thus violate our irreducibility hypothesis, we take $\tr(PR)=-2$, which fixes $r_{21}=-4$. 

    In this case, we have $\tr(QR)=\pm2$, and so
    \[\tr(QR)-\tr(R) = q_{12}r_{21} = -4q_{12} = \pm 2 - 2,\]
    implying either $q_{12}=1$ (and so $Q=P$) or $q_{12}=0$ (and so $Q=1$). In either case, the only conditions on $R$ are that $\tr(R)=+2$, $\det(R)=1$, and $\tr(PR)=-2$, so it can take the general form
    \[R=\begin{pmatrix}
        \,1+2s\phantom{.} & s^2\\ -4\phantom{.} & 1-2s
    \end{pmatrix},\qquad s\in\CC.\]
    All three of $\{P,Q,R\}$ then preserve the Hermitian form
    \begin{equation}\label{eq:H_weird}
        H=\begin{pmatrix}
        \,\Im(s)\phantom{.}& -i\\
        i\phantom{.} &
    \end{pmatrix},
    \end{equation}
    proving the theorem.
\end{proof}

We move on now to the case where the representation of $G$ on $\CC^2$ is \emph{reducible}, first showing that any unitary group must be conjugate to a subgroup of one of two model groups:
\begin{lemma}\label{lem:reducible}
    Suppose $G\subset\SL(2,\CC)$, and suppose the natural representation of $G$ on $\CC^2$ is reducible. In this case, $G$ is unitary if and only if it is conjugate to a subgroup of either $\SO(2)\subset\SL(2,\CC)$ or the group $\RR\rtimes\RR^\times\subset\SL(2,\RR)$ of real upper-triangular matrices. 
\end{lemma}
\begin{proof}
    First, suppose $G$ is unitary, and fix $P,Q\in G$. By the reducibility hypothesis, we can conjugate $G$ so that all $g\in G$ are upper-triangular. If $\tr(P)\neq\pm 2$, we can bring it into the form~\eqref{eq:eigenbasis} while retaining the upper-triangular form of $Q$.

    Suppose $G$ preserves a Hermitian form
    \begin{equation}\label{eq:hermitianform}
    H=\begin{pmatrix}
        \,h_{11} & h_{12}\\\,\ovl{h}_{12} & h_{22}
    \end{pmatrix}.
    \end{equation}
    If $|\tr(P)|>2$, matching elements of the equation $P^\dagger H = HP^{-1}$ shows that $h_{11}=h_{22}=0$; in particular, we can conjugate $G$ by the matrix~\eqref{eq:transform} to bring $H$ to the form~\eqref{eq:standard_form}, or a real multiple thereof. Proposition~\ref{prop:prop1} then shows that $G$ is conjugate to a subset of the group $\RR\rtimes\RR^\times\subset\SL(2,\CC)$ of real upper-triangular matrices.

    If $|\tr(P)|<2$, matching elements of $P^\dagger H=HP^{-1}$ instead shows that $h_{12}=0$, and subsequently matching elements of $Q^\dagger H = HQ^{-1}$ shows that $Q$ must take the form
    \[Q=\begin{pmatrix}
        \,\lambda_Q&\\&\lambda_Q^{-1}
    \end{pmatrix}.\]
    But from Lemma~\ref{lem:realtraces}, we know that $\tr(P)$, $\tr(Q)$, and $\tr(PQ)$ must be real; since $|\tr(P)|<2$, we have $\lambda_P\in S^1$, and these conditions show that $\lambda_Q\in S^1$ as well. Of course, the group of matrices of this form is simply an eigendecomposition of $\SO(2)$.

    The alternative is that $\tr(P)=\pm 2$ for every $P\in G$. Fix $P,Q\in G$ such that $P\neq\pm I$; if this is not possible, then $G$ is trivial. Noting that both of our model groups are closed under negation, suppose without loss of generality that $\tr(P)=\tr(Q)=+2$, so that
    \[P=\begin{pmatrix}\,1\phantom{.}& z_p\\& 1\end{pmatrix},\qquad Q=\begin{pmatrix}
        \,1\phantom{.} & z_q\\  & 1
    \end{pmatrix},\qquad z_p,z_q\in\CC.\]
    Moreover, conjugating by the matrix~\eqref{eq:transform} allows us to fix $z_p=1$. Suppose that $P$ and $Q$ preserve the form~\eqref{eq:hermitianform}. Matching elements yields $h_{11}=0$, implying $h_{12}\neq 0$, as well as $h_{12} = -\ovl{h}_{12}$ and $\ovl{z}_q h_{12}=-z_q\ovl{h}_{12}$; these conditions require $z_q\in\RR$, which implies our result. 
\end{proof}

Finally, we treat general subgroups of $\GL(2,\CC)$. In the reducible case, we use the classification of Lemma~\ref{lem:reducible} to deduce the dimension of $\RR\otimes G$ as a real algebra; since conjugation is an isomorphism of algebras, this property is basis-independent. To extend the irreducible case beyond the hypotheses of Proposition~\ref{prop:general}, we make use of the fact that its proof is constructive; given $P$, $Q$, and $R$, we can deduce the Hermitian form $H$ that they preserve. Indeed, even two of these matrices determine $H$ uniquely, so we can study a generic $R'\in G$ by looking at the group generated by $\{P,Q,R'\}$. We prove the following result; for convenience, we say that $G$ is \emph{reducible} if the natural representation of $G$ on $\CC^2$ is reducible, and \emph{irreducible} otherwise.
\begin{theorem}\label{thm:verygeneral}
    Suppose $G\subset\GL(2,\CC)$, and define the rescaled group
    \begin{equation}\label{eq:rescaled}
    G' = \left\{\pm g/\sqrt{\det g}\;\big|\; g\in G\right\}.
    \end{equation}
    Let $A$ be the real algebra generated by $G'$; if $G\subset\SL(2,\CC)$ already, we can equivalently consider $G'=G$. Then $G$ is unitary if and only if $|\det g|=1$ for all $g\in G$ and one of the following (mutually exclusive) criteria is satisfied:
    \begin{enumerate}
        \item $G\subset S^1$ is scalar; equivalently, $\dim(A)=1$.
        \item $G$ is abelian, and $\dim(A)=2$. Equivalently, $G'$ is conjugate to a non-scalar subgroup of $\SO(2)$, of the group $\RR^\times\subset\SL(2,\RR)$ of real diagonal matrices with unit determinant, or of the group $\ZZ_2\times\RR\subset\SL(2,\RR)$ of real upper-triangular matrices of trace $\pm 2$.
        \item $G$ is reducible but non-abelian, and $\dim(A)=3$. Equivalently, $G'$ is conjugate to a non-abelian subgroup of the group     $\RR^+\rtimes\RR^\times\subset\SL(2,\RR)$ of real upper-triangular matrices.
        \item $G$ is irreducible, and $\dim(A)=4$. Equivalently, $G$ is conjugate to an irreducible subgroup of either $\SL(2,\RR)$ or of $\U_2$. Equivalently, $G'$ is irreducible, and, for a generating set $S\subset G'$ and all $P,Q,R\in S$, we have \[\tr(P),\tr(Q),\tr(R),\tr(PQ),\tr(PR),\tr(QR)\in\RR\]
    and either $\tr(PQR)\in\RR$ or $\tr(PRQ)\in\RR$.
    \end{enumerate}
\end{theorem}
\begin{remark*}
    We will see that, if $G$ is irreducible and we can identify non-commuting elements $P,Q\in S$ with $|\tr(P)|\geq 2$, one can take $P$ and $Q$ to be fixed in the above trace conditions.
\end{remark*}
\begin{proof}
    If any $g\in G$ has $|\det g|\neq 1$, then it is clear that $G$ cannot be unitary; the proof of this fact follows the same lines as that of Proposition~\ref{prop:onlyreal}. We thus assume $|\det g|= 1$ for all $g\in G$. Furthermore, if $G$ is abelian, reducible, or unitary, then the rescaled group $G'$ shares the same property, and vice versa. As such, we suppose without loss of generality that $G\subset\SL(2,\CC)$, and handle each case in turn.

    \paragraph*{Abelian case.} Suppose first that $G$ is unitary, reducible, and abelian, but contains a non-scalar element; by Lemma~\ref{lem:reducible}, we know that $G$ is (up to conjugation) contained either in $\SO(2)\subset\SL(2,\CC)$, in the subgroup $\ZZ_2\times\RR\subset\SL(2,\CC)$ of real upper-triangular matrices of trace $\pm 2$, or in the subgroup $\RR^\times\subset\SL(2,\CC)$ of real diagonal matrices. All three generate an algebra of real dimension $2$. 

    Conversely, suppose $G$ is reducible and abelian and that $\dim(A)= 2$. Fix $P\in G$ with $P\neq\pm I$; if this is not possible, then $G$ is scalar. Now, if $\tr(P)\notin\RR$, then $P+P^{-1}=2\tr(P)$ would be linearly independent from $I$ over the reals; diagonalizing $P$, we thus see that $A$ would contain the ($4$ dimensional) algebra of complex diagonal matrices. As such, we know that $\tr(P)\in\RR$. If $\tr(P)\neq\pm 2$, we can write $P$ in the form~\eqref{eq:eigenbasis}; otherwise, we can write $P$ in the form~\eqref{eq:jordan}, up to sign. Now, the set $\{I,P\}$ already generates an algebra of real dimension $2$ in both cases. Indeed, if $|\tr(P)|>2$, this set spans the algebra of real diagonal matrices; if $|\tr(P)|<2$, it spans the same algebra as $\SO(2)$; and if $|\tr(P)|=2$, it spans the algebra of real upper-triangular matrices with repeated eigenvalues. By our hypothesis, then, $A$ (and thus $G$) is contained in this span. But Theorem~\ref{thm:verygeneral} shows that the intersection of any of these three algebras with $\SL(2,\CC)$ is unitary, so the same must be true of $G$.

    \paragraph*{Reducible, non-abelian case.}Next, suppose that $G$ is unitary, reducible, and non-abelian. From Lemma~\ref{lem:reducible}, we see that $G$ is contained (up to conjugation) in the group $\RR\rtimes\RR^\times\subset\SL(2,\CC)$ of real upper-triangular matrices, which generates an algebra of real dimension $3$. If we fix non-commuting $P,Q\in G$, then the algebra generated by $\{I,P,Q\}\subset G$ is of real dimension $\geq 3$, so it follows that $\dim(A) = 3$ exactly.

    Conversely, suppose that $G$ is reducible and non-abelian and that $\dim(A)= 3$. Fix non-commuting $P,Q\in G$. As before, we know that $\tr(P),\tr(Q),\tr(PQ)\in\RR$. If $\tr(P)\neq\pm 2$, we again write $P$ in the form~\eqref{eq:eigenbasis} and $Q$ in the form~\eqref{eq:QR_elements}, and the reducibility hypothesis allows us to fix $q_{21}=0$. If $|\tr(P)|>2$, the derived trace conditions show that the diagonals of both $P$ and $Q$ are real; if $|\tr(P)|<2$, they show that all diagonal elements are contained in $S^1$. In either case, the non-commuting condition requires $q_{12}\neq 0$, so, by conjugating $G$ appropriately, we fix $q_{12}=1$. As such, the real algebra $A'$ generated by $\{I,P,Q\}$ contains the one-dimensional space of real, \emph{strictly} upper-triangular matrices. If any of the diagonal elements of $P$ or $Q$ were not real, then $A'$ would also contain the space of \emph{complex}, strictly upper-triangular matrices, and we would find that $\dim(A')\geq 4$. Otherwise, $A'$ is exactly the algebra of real upper-triangular matrices, which is unitary by Proposition~\ref{prop:prop1}. The final case is that $\tr(P)=\pm 2$. We suppose without loss of generality that $\tr(P)=+2$, and write $P$ in the form~\eqref{eq:jordan}. For $P$ and $Q$ not to commute, then, we must have $q_{11},q_{22}\neq\pm 1$. If $q_{11},q_{22}\notin\RR$, the same argument as before would demonstrate that the real algebra $A'$ generated by $\{I,P,Q\}$ has $\dim(A')\geq 4$, and we would reach a contradiction. Thus, $q_{11},q_{22}\in\RR$, and the only way to satisfy $\dim(A')=3$ is to have $q_{12}\in\RR$ as well.
    
    \paragraph*{Irreducible case, trace conditions.} For the irreducible case, we first prove that unitarity is equivalent to the stated trace conditions. Fix non-commuting matrices $P,Q\in S$; since all irreducible representations of abelian groups are one-dimensional, such a choice must be possible. The forward direction (unitary$\implies$real traces) follows from Lemma~\ref{lem:realtraces}, so we focus on the converse.

    Suppose that $|\tr(P)|\neq 2$, and write $P$ in the form~\eqref{eq:eigenbasis} and $Q$ in the form~\eqref{eq:QR_elements}. If $|\tr(P)|>2$, then the trace conditions on $P$ and $Q$ guarantee that $\lambda_P,q_{11},q_{22}\in\RR$, and, with an additional conjugation by the matrix~\eqref{eq:transform}, we can further ensure that $q_{12},q_{21}\in\RR$. This is clear if one of the two elements vanishes---note that they cannot both vanish, or $P$ and $Q$ would commute. If neither element vanishes, we can choose $q_{12}\in\RR$ and apply the condition $\det(Q)=1$ to deduce $q_{21}\in\RR$. Without loss of generality, we suppose $q_{12}\neq 0$.

    Choose an $R\in S$, and note that $r_{11},r_{22}\in\RR$. Since $q_{12}\neq 0$, the trace conditions $\tr(QR),\tr(PQR)\in\RR$ (or $\tr(QR),\tr(QPR)\in\RR$) show that $q_{12}r_{21},q_{21}r_{12}\in\RR$ and thus that $r_{21}\in\RR$. If $r_{21}\neq 0$, the condition $\det(R)=1$ thus implies that $r_{12}\in\RR$. Otherwise, fix a matrix
    \[M=\begin{pmatrix}
        \,m_{11}&m_{12}\\\,m_{21}&m_{22}
    \end{pmatrix}\in S\]
    with $m_{21}\neq 0$; the existence of such a matrix follows from our irreducibility hypothesis. As before, we find that $m_{11},m_{22},m_{12},m_{21}\in\RR$, so the trace conditions $\tr(MR),\tr(PMR)\in\RR$ (or $\tr(MR),\tr(MPR)\in\RR$) imply that $r_{12}\in\RR$. Since $R$ is general, and our conjugation depends only on the matrices $P$ and $Q$, we find $G\subset\SL(2,\RR)$, so Proposition~\ref{prop:prop1} shows that $G$ is unitary.

    Next, if $|\tr(P)|<2$, the same trace conditions show that $\lambda_P\in S^1$ and $q_{11}=\ovl{q}_{22}$. Choose an $R\in S$; then the trace conditions now show that $r_{11}=\ovl{r}_{22}$ and $q_{21}r_{12}=\ovl{q_{12}r_{21}}$. If $q_{12}$ (resp., $q_{21}$) were to vanish, we would find that $r_{12}$ (resp., $r_{21}$) must vanish as well; but $R$ was generic, so this would violate our irreducibility hypothesis. Thus, $q_{12},q_{21}\neq 0$, and the set $\{P,Q,R\}$ preserves the Hermitian form~\eqref{eq:hermitianform} with $h_{21}=0$, $h_{11}=1$, and $h_{22}=-q_{12}/\ovl{q}_{21}\in\RR$. Since our conjugation of $G$ and the choice of Hermitian form $H$ depend only on $P$ and $Q$, unitarity follows.

    The final case is, whenever $P,Q\in S$ are non-commuting, we have $\tr(P),\tr(Q)=\pm 2$; without loss of generality, suppose $\tr(P)=\tr(Q)=2$. By hypothesis, $P,Q\neq I$, so we can bring $P$ to the form~\eqref{eq:jordan}. In this basis, we must have $q_{21}\neq 0$; otherwise, the conditions $\det(Q)=1$ and $\tr(Q)=2$ would imply $q_{11}=q_{22}=1$, which would violate the non-commuting assumption. 

    Now, $P$ also cannot commute with $PQ$, so we must have $\tr(PQ)=\pm 2$. As in Proposition~\ref{prop:general}, the non-commuting condition requires $\tr(PQ)=-2$, and thus $q_{21}=-4$; then $Q$ takes the following form:
    \[Q=\begin{pmatrix}
        \,1+2s\phantom{.} & s^2\\ -4\phantom{.} & 1-2s
    \end{pmatrix},\qquad s\in\CC.
    \]
    Now, $P$ and $Q$ uniquely preserve the Hermitian form~\eqref{eq:H_weird}. However, Proposition~\ref{prop:general} shows that, for any $R\in S$, the group generated by $\{P,Q,R\}$ is unitary. Thus, $R$ preserves the same Hermitian form, and $G$ is unitary.

    \paragraph*{Irreducible case, dimension condition.} Finally, we prove that, if $G$ is irreducible, then it is unitary if and only if $\dim(A)=4$. Suppose that $G$ is irreducible and unitary; by an appropriate conjugation, we can ensure either that $G\subset\SL(2,\RR)$ or that $G\subset\SU_2$, and so $\dim(A)\leq 4$. By the density theorem~\cite[Section~3.2]{etingof2011introduction}, we know that $G$ must generate the full matrix algebra $\CC^{2\times 2}$ over the complex numbers; since this algebra is of (complex) dimension 4, the real algebra $A$ must be of real dimension $\geq 4$. Alternatively, we can derive this result more directly. In the case that $G\subset\SL(2,\RR)$, the only way for $\dim(A)<4$ would be for all $P\in G$ to be upper-triangular (or lower-triangular), which would violate irreducibility. In the case that $G\subset\SU_2$, we fix non-commuting $P,Q\in G$ and conjugate by a \emph{unitary} matrix to ensure $P$ takes the form~\eqref{eq:eigenbasis} and that $G$ is still contained in $\SU_2$; that this is possible follows from the spectral theorem. The trace conditions $\tr(Q),\tr(PQ)\in\RR$ then imply that $q_{11}=\ovl{q}_{22}$, so we can subtract off the diagonal of both $Q$ and $PQ$ using an appropriate linear combination of $\{I,P\}$; but $\lambda_P\neq\pm 1$, so the algebra $A'$ generated by $\{I,P,Q\}$ must contain all matrices of the following form:
    \[Z=\begin{pmatrix}
        \phantom{-}z\phantom{.} & w\\-\ovl{w}\phantom{.}&\ovl{z}
    \end{pmatrix}.\]
    Since $A'\subset A$, this implies that $\dim(A)=4$. We note that, in this latter case, $A$ is isomorphic to the quaternion algebra.

    Conversely, suppose that $G$ is irreducible and that $\dim(A)=4$. If any $P\in G$ were to satisfy $\tr(P)\notin\RR$, then $P+P^{-1}=2\tr(P)I$ would be linearly independent from $I$ over the reals; but this would imply that the algebra $A'$ generated by $\{I,P\}$ already has $\dim(A')=4$, so $A'=A$. But then, any $Q\in G$ would have to be contained in this algebra, so, in particular, $G$ would have to be reducible. As such, we know that $\tr(P)\in\RR$ for all $P\in G$, which implies unitarity.
\end{proof}

\section{Unitary Monodromy of Heun and Darboux Operators} \label{sec:HeunAnalysis}
We have approached the problem~\eqref{qs:app} in the general case with Theorem~\ref{thm:verygeneral}, in the sense that, once the monodromy group $G\subset\GL(2,\CC)$ is known, Theorem~\ref{thm:verygeneral} provides a necessary and sufficient criterion for its unitarity. We now discuss how our result applies to the Heun~\eqref{eq:Heun} and Darboux~\eqref{eq:Darboux} operators specifically. In so doing, we will be able to make progress towards characterizing the \emph{point-spectrum} (i.e., set of eigenvalues) of Heun and Darboux operators.

For one, since the monodromy matrices of the Darboux equation are all of unit determinant, we can apply Theorem~\ref{thm:verygeneral} with $G'=G$:

\begin{corollary}\label{cor:general_darboux}
    Let $G$ be the monodromy group of the Darboux equation~\eqref{eq:Darboux}. If the natural representation of $G$ on $\CC^2$ is reducible, then $G$ is unitary if and only if it is conjugate to a subgroup of either $\SO(2)$ or the group $\RR\rtimes\RR^\times\subset\SL(2,\RR)$ of real upper-triangular matrices. 

    If the representation of $G$ on $\CC^2$ is irreducible, fix local monodromy matrices $P,Q,R\in G$. Then $G$ is unitary if and only if the following conditions are met:
    \begin{enumerate}
        \item For each $i\in\{0,1,2,3\}$, either $m_i\in\RR$ or $\Re(m_i)\in\frac{1}{2}\ZZ$.
        \item The traces $\tr(PQ)$, $\tr(QR)$, and $\tr(PR)$ are real, or equivalently, the matrices $\{P,Q,R\}$ generate a real algebra of dimension $4$.
    \end{enumerate}
\end{corollary}
\begin{proof}
    The reducible case follows directly from Theorem~\ref{thm:verygeneral}. For the irreducible case, one need only verify that the trace conditions $\tr(P),\tr(Q),\tr(R),\tr(PQR)\in\RR$ are equivalent to our stated conditions for the parameters $m_i$. Using the expressions~\eqref{eq:localexp_darboux} for the local exponents, we calculate
    \[\tr (P) = e^{-2\pi i m_1} + e^{2\pi i (1+m_1)} = 2\cos(2\pi m_1),\]
    which is real if and only if $m_1\in\RR$ or $\Re(m_1)\in\frac{1}{2}\ZZ$. The calculations for $Q$, $R$, and $PQR$ yield the same for $m_2$, $m_3$, and $m_0$, respectively.
\end{proof}

Before applying Theorem~\ref{thm:verygeneral} to the Heun equation~\eqref{eq:Heun}, we note that---except for particular choices of parameters---the representation of $G$ on $\CC^2$ is necessarily irreducible in this case. Indeed, suppose it is reducible, and suppose $u\in\CC^2$ is a shared eigenvector of all matrices $g\in G$. Fix local monodromy matrices $P$, $Q$, and $R$ about $x=0$, $1$, and $a$, respectively. On one hand, the expressions~\eqref{eq:localexp_heun} show that the eigenvalue of $PQR$ associated with $u$ must be either $e^{2\pi i\alpha}$ or $e^{2\pi i\beta}$. On the other hand, we know that the eigenvalue of $P$ associated with $u$ must be either $1$ or $e^{2\pi i\gamma}$, and similarly with $Q$ and $R$; multiplying these matrices together and recalling that $\gamma+\delta+\eps=1+\alpha+\beta$, we see that the representation can {only} be reducible if
\begin{equation}\label{eq:heun_reducible}
    \alpha, \beta\in\{0,\gamma,\delta,\varepsilon,\gamma+\delta,\gamma+\varepsilon,\delta+\varepsilon,\gamma+\delta+\varepsilon\}\pmod{1}.
\end{equation}
In particular, this criterion ensures that the representation is irreducible in the Lam\'e case~\eqref{eq:lameAlgebraic} with $A=1/4$, where $\gamma=\delta=\eps=1/2$ and $\alpha=\beta=1/4$; in this case, the above argument reduces to one of Beukers~\cite{beukers2007unitary}. It similarly ensures irreducibility if $\alpha,\beta\notin\RR$. In any case, we can apply Theorem~\ref{thm:verygeneral} once again:

\begin{corollary}\label{cor:general_heun}
    Let $G$ be the monodromy group of the Heun equation~\eqref{eq:Heun}, and suppose~\eqref{eq:heun_reducible} is not satisfied. Fix local monodromy matrices $P,Q,R\in G$ and  define $P_0=e^{\pi i\gamma}P$, $Q_0=e^{\pi i\delta}Q$, and $R_0=e^{\pi i\varepsilon}R$. Then $G$ is unitary if and only if the following conditions are met:
    \begin{enumerate}
        \item The parameters $\gamma$, $\delta$, and $\varepsilon$ are real, and either $\alpha,\beta\in\RR$ or $\alpha-\ovl{\beta}\in\ZZ$.
        \item The traces $\tr(P_0Q_0)$, $\tr(Q_0R_0)$, and $\tr(P_0R_0)$ are real, or equivalently, the matrices $P_0$, $Q_0$, and $R_0$ generate a real algebra of dimension 4.
    \end{enumerate}
\end{corollary}

It is worth noting that the setting of Corollary~\ref{cor:general_darboux} is somewhat more general than that of Corollary~\ref{cor:general_heun}. In particular, the Heun parameters $\gamma,\delta,\varepsilon\in\RR$ correspond to Darboux parameters $m_1,m_2,m_3\in\RR$, as we can see from the transformations~\eqref{eq:params_D2H}, and the restrictions on $\alpha$ and $\beta$ correspond to the parameter space $m_0\in\RR\cup\left(\frac{1}{2}+\ZZ + i\RR\right)$. By contrast, our result on the Darboux equation applies to the larger parameter space $m_i\in\RR\cup\left(\frac{1}{2}\ZZ + i\RR\right)$.

To make sense of this, recall that the variable $z$ of~\eqref{eq:Darboux} and the variable $x$ of~\eqref{eq:Heun} are related as $x\sim z^2$ about every pole. If $G_H$ is the monodromy group of a fixed Heun equation and $G_D$ is that of the corresponding Darboux equation, then so long as both act irreducibly on $\CC^2$, we see that the resulting map $G_D\hookrightarrow G_H$ is a \emph{strict} inclusion; a simple closed loop in $G_D$ corresponds to {two} loops in $G_H$. As a result, if $G_H$ is unitary, we expect the same of $G_D$; if only $G_D$ is unitary, however, we cannot generally expect the same of $G_H$. 

We can make more progress on this question by investigating the point-spectrum of Heun and Darboux operators. To this end, we transcribe Proposition~\ref{prop:prop3} below from~\cite[Prop.~4]{beukers2007unitary}: 
\begin{proposition}
\label{prop:prop3}
Let $G$ be the monodromy group of the linear second order differential equation $y'' + py' + qy = 0$, where $p,q \in \CC(z)$. Then $G$ is unitary if and only if there exists a nonzero $C^2$ function $f$ on $\CC\setminus \{0,1,a\}$ which is a real-analytic solution of the equation. Furthermore, if the natural representation of $G$ on $\CC^2$ is irreducible, then $f$ is uniquely determined up to a constant factor. \end{proposition}
\begin{remark*}
    Although stated in terms of rational coefficients $p$ and $q$, Beukers' proof holds for the more general case where $p$ and $q$ are simply meromorphic. In particular, it applies equally well to the Darboux equation~\eqref{eq:Darboux}.
\end{remark*}

With this result in hand, we see that Theorem~\ref{thm:verygeneral} provides a criterion for any $\lambda\in\CC$ to belong to the point-spectrum of any second-order operator, based on the monodromy group corresponding to the equation $y'' + py' + qy = \lambda y$. We emphasize again that computing the monodromy group is typically a highly transcendental step; even still, we will be able to recover practical numerical results from this line of reasoning in the following section. More immediately, Proposition~\ref{prop:prop3} allows us to connect the point-spectra of Heun and Darboux operators by slightly modifying the coordinate transform of Section~\ref{sec:Background_eqs}.

As stated, the change of variables does not take one real-analytic function to another, as it maps
\[y(x)\mapsto x^{-m_1/2}(x-1)^{-m_2/2}(x-a)^{-m_3/2}y(x) =: u\]
for real values $m_1$, $m_2$, and $m_3$; however, we can correct this behavior by instead mapping
\begin{equation}\label{eq:corrected_variable}
    y(x) \mapsto |x|^{-m_1}|x-1|^{-m_2}|x-a|^{-m_3}y(x),
\end{equation}
which differs from the above by only an antiholomorphic factor (and thus continues to satisfy the Darboux equation). As such, Proposition~\ref{prop:prop3} yields the following corollary:
\begin{corollary}\label{cor:RDarbouxSpec}
    If a Heun equation~\eqref{eq:Heun} carries a unitary monodromy group with irreducible action on $\CC^2$, then the corresponding Darboux equation~\eqref{eq:Darboux}---resulting from the transform outlined in Section~\ref{sec:Background_eqs}---also carries a unitary monodromy group.
\end{corollary}

The converse does \emph{not} hold, as discussed above: not every real-analytic solution of the Darboux equation corresponds to a real-analytic solution for the Heun equation. In the change of variables taken to reach the Darboux equation, note that $z$ and $-z$ have the same pre-image $x = (\wp(z)-e_1)/(e_2-e_1)$. In moving from the Heun to the Darboux equation, then, we introduce an additional requirement that the solution is \emph{even}---in general, solutions of the Heun equation~\eqref{eq:Heun} correspond to even solutions of the Darboux equation~\eqref{eq:Darboux}.

\section{Estimating the Eigenvalues of Heun and Darboux Operators} \label{sec:asymp}
Following Beukers~\cite{beukers2007unitary}, we provide an asymptotic estimate for the eigenvalues of the Heun equation, and we compare this estimate to a numerical approximation of the same eigenvalues. In so doing, we argue that Beukers' derived estimate should extend to any second-order equation with at most four regular singularities and provide evidence that our Theorem~\ref{thm:verygeneral} can be strengthened to require fewer trace conditions.

As a starting point, we convert the Heun equation~\eqref{eq:Heun} into the Darboux equation~\eqref{eq:Darboux} through the change of variable defined in Section~\ref{sec:Background_eqs}; this procedure yields a lattice $\Lambda\subset\CC$ with half-periods $\omega_1$, $\omega_2$, and $\omega_3 = \omega_1 + \omega_2$, and parameters $m_0\in\RR\cup(\frac{1}{2} +\ZZ + i\RR)$ and $m_1,m_2,m_3\in\RR$. Notably, these parameters satisfy the first criterion of Corollary~\ref{cor:general_darboux}, so we might expect to find unitary monodromy---or equivalently, single-valued real-analytic solutions (Proposition~\ref{prop:prop3}).

In this setting, we extend and revise Beukers' conjecture on the distribution of eigenvalues $B$ of the Heun operator. Importantly, recall from Section~\ref{sec:Background_eqs} that these values are distinct from the eigenvalues $B'$ of the Darboux operator; if such a $B$ exists, then a corresponding $B'$ exists (Corollary~\ref{cor:RDarbouxSpec}) and is given by $B' = (\wp(\omega_3) - \wp(\omega_1))B + B_0$, where a full expression for $B_0 = B_0(\alpha,\beta,\gamma,\delta,\eps)$ is given in \cite{2001math.....12179T}. These values $B'$ correspond to single-valued, real-analytic, \emph{even} solutions of the Darboux equation.

Beukers' conjecture \cite[Conj.~1]{beukers2007unitary} is transcribed below, with slightly modified notation:

\begin{conjecture}[Beukers, 2007]\label{conj:beukers}
    Let $\ovl{\Lambda}$ be the lattice generated by $2\ovl{\omega}_1$ and $2\ovl{\omega}_2$, and let 
    \[\Delta = 4\operatorname{Im}(\ovl{\omega}_1\omega_2) = 2i(\omega_1\ovl{\omega}_2 - \omega_2\ovl{\omega}_1)\]
    be the area of its fundamental parallelogram. Furthermore, let 
\[\zeta(z) = \frac{1}{z} + \sum_{w\in\Lambda\setminus\{0\}}\left(\frac{1}{z-w} + \frac{1}{w} + \frac{z}{w^2}\right)\]
be the Weierstrass zeta function~\cite{cartan1995elementary} on $\Lambda$, and let $\eta_i = \zeta(z + 2\omega_i) - \zeta(z)$ be its quasi-periods. Then the accessory parameters $B'$ such that the Lam\'e equation \eqref{eq:Darboux_Lame} carries a single-valued, real-analytic, even solution are as follows:
    \[B' = \ell_0^2- \frac{2}{i\Delta}m_0 (m_0+1)\left(\eta_1\overline{\omega_2}-\eta_2\overline{\omega_1} + \pi i\frac{\ell_0}{\overline{\ell_0}}\right) + \mathcal{O}(1/|\ell_0|),\]
    where $\ell_0\in\frac{2\pi}{\Delta}\ovl{\Lambda}$.
\end{conjecture}
\begin{remark*}
    Beukers' estimate is a special case of the more general WKB approximation~\cite{hall2013quantum,simon2015comprehensive}; notably, a similar approach was applied to a generalization of our setting by Gaiotto et al.~\cite{gaiotto2011wallcrossinghitchinsystemswkb}.
\end{remark*}

We split the modified conjecture into two distinct claims. First, we note that the argument underlying \cite[Conj.~1]{beukers2007unitary} can be straightforwardly adapted to support the following conjecture:
\begin{conjecture}\label{prop:asymptoticlattice}
Let $\ovl{\Lambda}$, $\Delta$, $\zeta(z)$, and $\eta_i$ be defined as in Conjecture~\ref{conj:beukers}. Let $\ell_0 \in \frac{2\pi}{\Delta}\overline{\Lambda}$, and consider the limit $\ell_0\to\infty$.
Fix an accessory parameter $B'$ such that the Darboux equation~\eqref{eq:Darboux} carries a single-valued, real-analytic, even solution. If $B'$ satisfies $\sqrt{B'} = \ell_0 + \eps$ for some $\eps\in (-\pi/\Delta,\pi/\Delta]\times (-\pi/\Delta,\pi/\Delta]$, then it must be given by
    \[B' = \ell_0^2- \frac{2}{i\Delta}\left(\sum_{i=0}^3 m_i (m_i+1)\right)\left(\eta_1\overline{\omega_2}-\eta_2\overline{\omega_1} + \pi i\frac{\ell_0}{\overline{\ell_0}}\right) + \mathcal{O}(1/|\ell_0|).\]
\end{conjecture}
The claim of this first conjecture is somewhat weaker than that of Conjecture~\ref{conj:beukers}. Namely, it does not make any claim that each lattice point in $\ovl{\Lambda}$ must be associated to a unique eigenvalue of the Heun operator; rather, it only claims that, if we know \emph{a priori} that a Heun eigenvalue lies near a given lattice point, then it should be offset by the predicted asymptotic formula.

In general, this is all one should expect from the sort of asymptotic argument used to justify Conjectures~\ref{conj:beukers} and~\ref{prop:asymptoticlattice}; the WKB approximation is not generally expected to converge to a precise answer, but it may provide an estimate of possible error for nearby solutions. One might then expect that, depending on the Heun equation, certain lattice points may have two or more associated eigenvalues, or none at all. On the other hand, by providing a rigorous proof of Conjecture~\ref{conj:beukers} in the particular case $m_0=1$, Beukers also suggests that the following, stronger claim may be true:
\begin{conjecture}\label{conj:conj2}
    There is a value $R>0$, possibly depending on the parameters $(m_0,m_1,m_2,m_3)$, such that for any $\ell_0 \in \frac{2\pi}{\Delta}\overline{\Lambda}$ with $|\ell_0|>R$, a unique accessory parameter $B'$ exists such that $|B'-\ell_0^2|\leq \ell_0$ and such that the Darboux equation carries a single-valued, real-analytic, even solution. Moreover, $B'$ satisfies the asymptotic estimate given by Conjecture~\ref{prop:asymptoticlattice}.
\end{conjecture}

We highlight one subtlety with the two conjectures we have raised: Corollaries~\ref{cor:general_darboux} and~\ref{cor:general_heun} appear to give three real algebraic constraints on our problem---corresponding to each of the traces $\tr(PQ)$, $\tr(QR)$, and $\tr(PR)$ being real---but the accessory parameter $B$ is only a single complex parameter, with two real degrees of freedom. In the case of the Lam\'e equation, this mismatch is handled naturally by the result of Beukers (Proposition~\ref{prop:beukers}); if any two of the pairwise traces are real, the third follows. Unitarity thus corresponds to two real algebraic constraints on two real parameters, and one expects a discrete set of solutions (as predicted in \cite[Conj.~1]{beukers2007unitary}).

We offer two pieces of evidence that a similar situation should hold in the general case, i.e., that Proposition~\ref{prop:general} (and thus Corollaries~\ref{cor:general_darboux} and~\ref{cor:general_heun}) can be strengthened to require only two pairwise traces to be real. First, we note that the argument underlying Conjecture~\ref{prop:asymptoticlattice} relies on the equivalence between unitary monodromy and the existence of real-analytic solutions (Proposition~\ref{prop:prop3}), rather than on our trace conditions above. In turn, the WKB approach used to approximate these real-analytic solutions gives rise to two real algebraic conditions at each order of approximation, just as in the Lam\'e case studied by Beukers.

Secondly, as we show shortly, numerical approximation offers direct evidence for two facts:
\begin{enumerate}
    \item The asymptotic values predicted in Conjecture~\ref{conj:conj2} correspond to unitary monodromy groups, despite the apparent mismatch of algebraic constraints discussed above.
    \item If one ensures that the traces $\tr(P_0Q_0)$ and $\tr(Q_0R_0)$ are real in the setting of Corollary~\ref{cor:general_heun}, then one consistently finds that the trace $\tr(P_0R_0)$ is real within a small tolerance.
\end{enumerate}


To investigate these claims, we leverage our analytical results to numerically approximate eigenvalues of the Heun equation. In short, Corollary~\ref{cor:general_heun} (and thus Theorem~\ref{thm:verygeneral} itself) transforms the accessory parameter problem~\eqref{qs:app2} into the following rootfinding problem:
\begin{equation}\label{eq:rootfinding}
\Im(\tr(P_0Q_0))=\Im(\tr(P_0R_0))=\Im(\tr(R_0Q_0))=0,
\end{equation}
where $P_0$, $Q_0$, and $R_0$ are the rescaled local monodromy matrices defined in Corollary~\ref{cor:general_heun}. 

A naive approach to this problem is to try to minimize the sum of squares of these three expressions. However, the resulting energy landscape does not lend itself well to traditional minimization schemes; it carries a multitude of nonzero local minima, for one, along with fine-scale structure in the vicinity of each true eigenvalue. While preconditioning may recover such an approach, it is outside the present scope.


\begin{figure}
    \centering
    \includegraphics[width=0.8\linewidth]{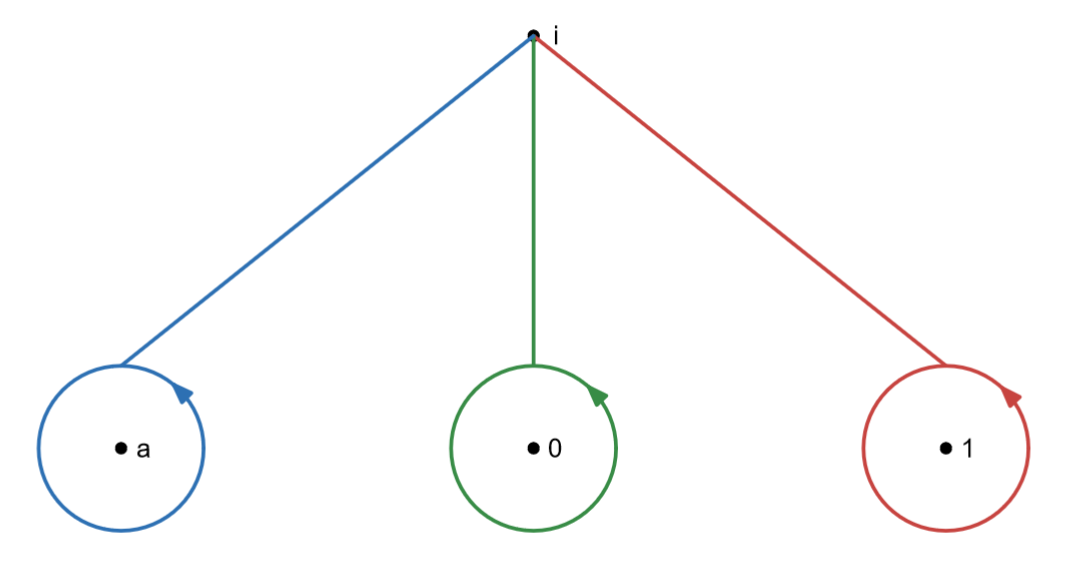}
    \caption{The three contours about which we integrate the Heun equation~\eqref{eq:Heun} in order to recover its monodromy matrices. We begin at the common point $i=\sqrt{-1}$, travel to the top-most point of a radius $1/5$ circle about each pole (fixing $a=-1$), and travel counter-clockwise around the circle before returning to our starting point.}
    \label{fig:contour}
\end{figure}

Instead, we use a Newton-like approach to find simultaneous roots of two of these three expressions, following the proposal of Beukers for the Lamé equation~\cite{beukers2007unitary}. First, we fix the third puncture of the Heun equation~\eqref{eq:Heun} at $a=-1$, and we choose a base point (here, we fix $z_0=i$) from which to integrate the equation. We numerically integrate the equation twice about each pole, following the contours drawn in Figure~\ref{fig:contour}; we first integrate with the initial conditions $y(x_0)=1$, $\frac{dy}{dx}(x_0)=0$, and then with the initial conditions $y(x_0)=0$, $\frac{dy}{dx}(x_0)=1$. We use a fourth-order Runge-Kutta method~\cite{butcher2008numerical} for the integration, with step size $|dx|=4\times 10^{-4}$. Sampling each solution at the end of the integration contour allows us to construct the monodromy matrices $P$, $Q$, and $R$, and rescaling by their (square-root) determinants yields the matrices $P_0$, $Q_0$, and $R_0$.

With these matrices in hand, we define the traces $t_{PQ}(B_n)=\tr(P_0Q_0)$ and $t_{QR}(B_n)=\tr(Q_0R_0)$ and, given a small $h\in\CC$, their approximate (holomorphic) derivatives
\[\nu = \frac{t_{PQ}(B_n + h)-t_{PQ}(B_n)}{h},\qquad \mu = \frac{t_{QR}(B_n + h)-t_{QR}(B_n)}{h}.\] 
In our own simulations, we fix $h=10^{-5}\in\RR$. To calculate our next estimate $B_{n+1} = B_n + \epsilon$, we want to ensure that $t_{PQ}(B_n + \epsilon)$ and $t_{QR}(B_n + \epsilon)$ are real. To first order, this criterion becomes
    \[\Im\left( t_{PQ}(B_n) + \nu \epsilon\right) = 0, \quad \Im \left( t_{QR}(B_n) + \mu \epsilon\right) = 0.\]
    Simultaneously solving both gives $\epsilon$ as
    \[\epsilon = \frac{ \overline{\mu} \Im(t_{PQ}(B_n)) - \overline{\nu} \Im(t_{QR}(B_n))}{\Im(\overline{\nu} \mu)}.\]
    Finally, we iterate $B$ as $B_{n+1} = B_n + \epsilon$ and repeat until convergence. If the third trace $t_{PR}(B)=\tr(P_0R_0)$ is real within a given tolerance (here, $3\%$ relative error), we accept the converged value.

\begin{figure}
    \centering
    \begin{subfigure}[m]{.63\textwidth}
        \centering
        \includegraphics[scale=0.47]{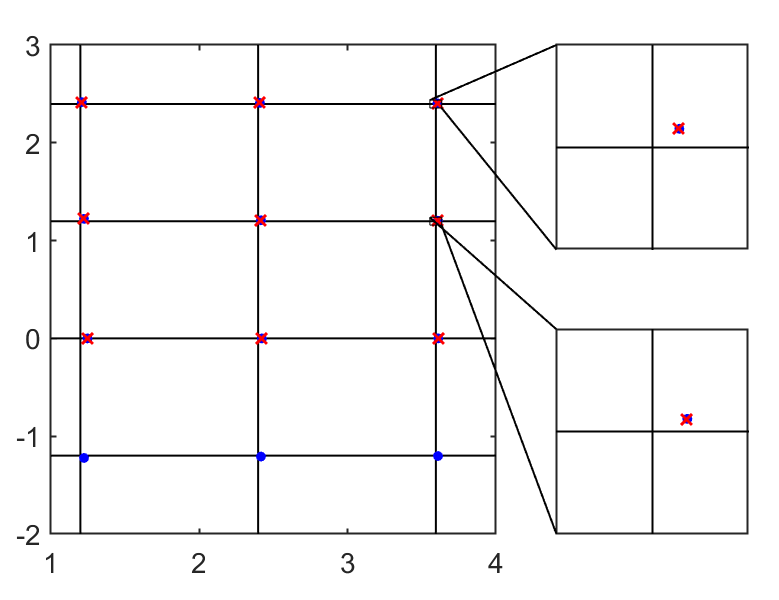}
        \caption{}\label{fig:LameFigure:a}
    \end{subfigure}%
    \begin{subfigure}[m]{.37\textwidth}
        \centering
        \includegraphics[scale=0.16]{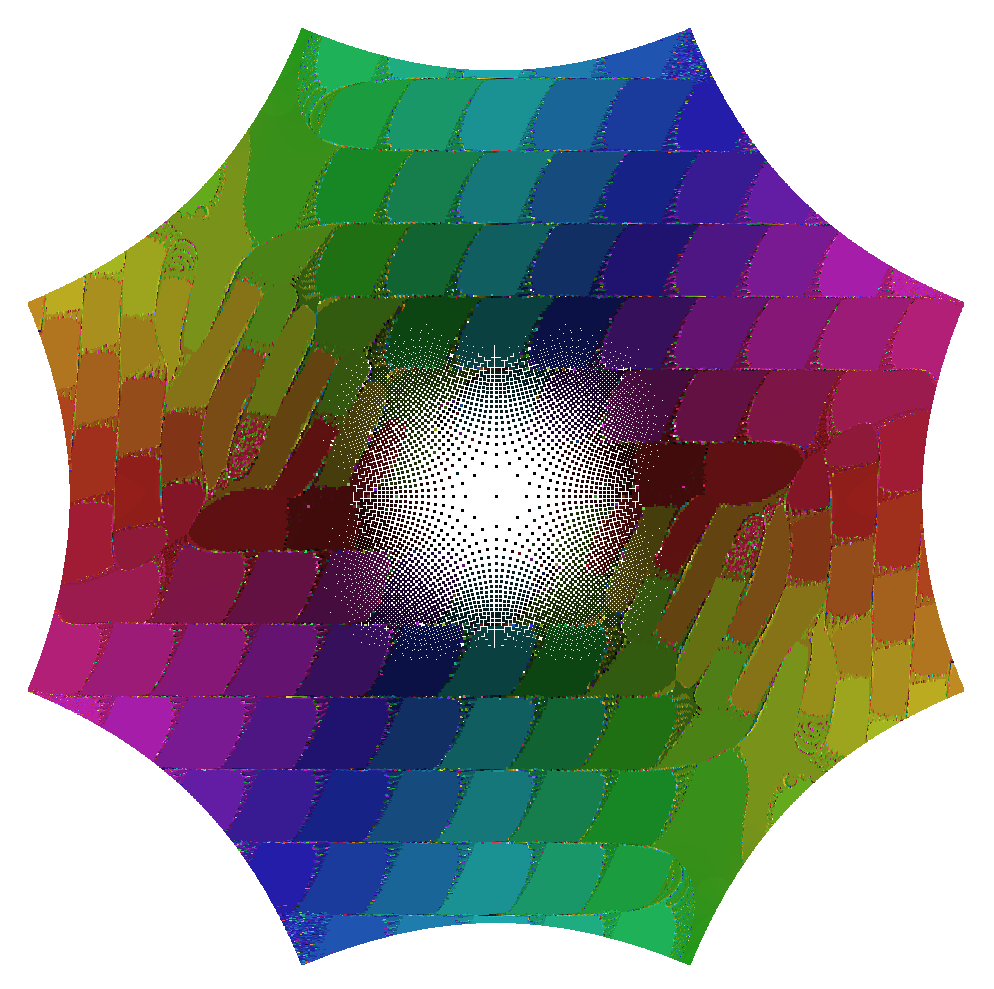}
        \caption{}\label{fig:LameFigure:b}
    \end{subfigure}
    \caption{\textbf{(a)} Approximate square-roots of the eigenvalues of the Lam\'e equation~\eqref{eq:lameAlgebraic}, found by starting our algorithm at each point $z_\text{init.}$ in the set~\eqref{eq:startingpoints} and iterating until convergence. All values satisfy~\eqref{eq:rootfinding} within a relative error of $1\%$. We show the lattice $\frac{2\pi}{\Delta}\ovl{\Lambda}\approx 1.198\,\ZZ^2$ in black and several asymptotic values (provided by Conjectures~\ref{conj:beukers} and~\ref{prop:asymptoticlattice}) in red; these are exactly the values considered by Beukers~\cite{beukers2007unitary}. In accordance with his conjecture, these eigenvalues appear to map one-to-one with lattice points; we zoom into a square of width $0.04$ about two lattice points in order to demonstrate this close agreement.  \textbf{(b)} A ``convergence map'' for the Lam\'e equation, showing the convergence properties of our algorithm in the (square root of the) region $[-50,50] + [-50,50]i\subset\CC$. The position of each pixel represents the initial value of $B$ in our algorithm, and its color (using a standard HSV color transform) represents the final value after 20 iterations. We see that our algorithm breaks down along certain complex arguments as we move further away from the origin.}
    \label{fig:LameFigure}
\end{figure}

\begin{figure}
    \centering
    \begin{subfigure}[m]{.33\textwidth}
        \centering
        \includegraphics[scale=0.35]{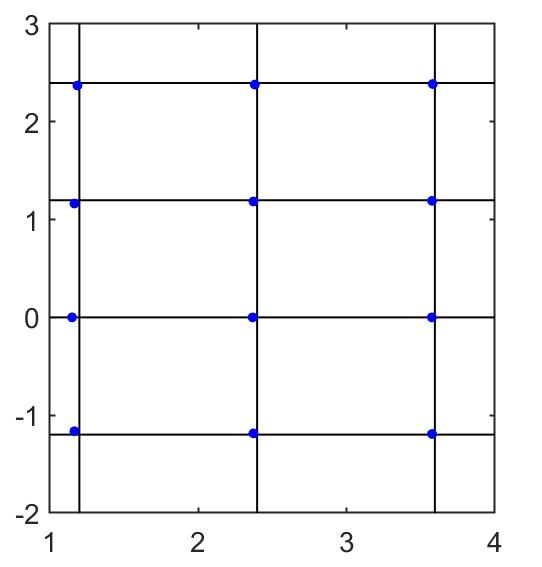}
        \caption{$\gamma=\eps=\delta=1/3$}\label{fig:HeunFigure:a}
    \end{subfigure}%
    \begin{subfigure}[m]{.33\textwidth}
        \centering
        \includegraphics[scale=0.35]{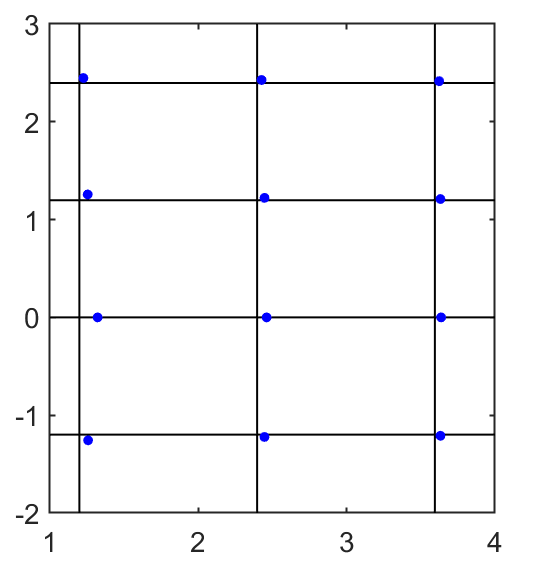}
        \caption{$\gamma=\eps=\delta=2/3$}\label{fig:HeunFigure:a}
    \end{subfigure}%
    \begin{subfigure}[m]{.33\textwidth}
        \centering
        \includegraphics[scale=0.35]{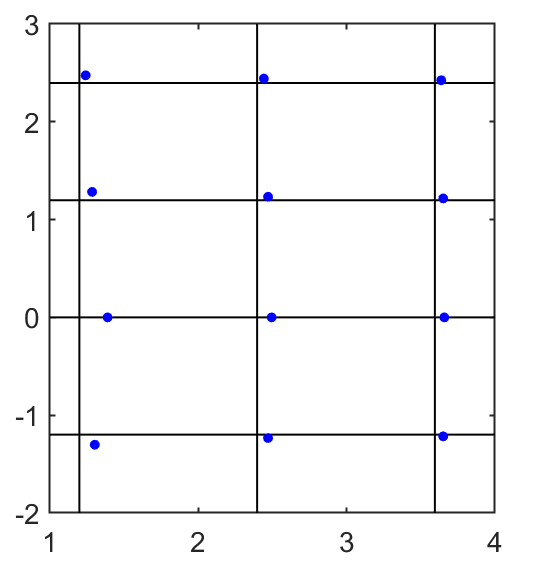}
        \caption{$\gamma=\eps=\delta=1$}\label{fig:HeunFigure:c}
    \end{subfigure}
    \caption{Approximate square-roots of the eigenvalues of three Heun equations, all with $\alpha=\beta$. As parameters move away from the Lam\'e values ($\gamma=\delta=\eps=1/2$), eigenvalues drift away from the lattice $\frac{2\pi}{\Delta}\Lambda$ in different directions.}
    \label{fig:HeunFigure}
\end{figure}
We apply our algorithm first to the Lam\'e equation (i.e., $\gamma=\delta=\eps=1/2$), where our algorithm reduces to that of Beukers~\cite{beukers2007unitary}. To exhaustively find all eigenvalues in a large region, we start our algorithm independently at each point
\begin{equation}\label{eq:startingpoints}
    \sqrt{z_\text{init.}}/1.198 \in \{1,2,3\} + \{-i,0,i,2i\}\subset\CC
\end{equation}
and run it until convergence; the factor of $1.198$ was chosen to align with the leading order prediction $\frac{2\pi}{\Delta}\ovl{\Lambda}\approx 1.198\,\ZZ^2$ of Proposition~\ref{prop:asymptoticlattice}. The square roots of our converged eigenvalues are shown Figure~\ref{fig:LameFigure:a} in blue, along with the lattice $\frac{2\pi}{\Delta}\ovl{\Lambda}$ in black and several asymptotic solutions in red, provided by Proposition~\ref{prop:asymptoticlattice}; the latter are exactly those of Beukers. We zoom into two example lattice points to demonstrate how the Lam\'e spectrum holds closely to the predicted, asymptotic values, consistent with Conjectures~\ref{conj:beukers},\ref{prop:asymptoticlattice}, and \ref{conj:conj2}.

Figure~\ref{fig:LameFigure:b} shows a ``convergence map'' of the same experiment, to clarify the global convergence properties of our algorithm. The position of each pixel represents an initial estimate of $\sqrt{B}$, and its color/brightness represents the converged value according to a standard HSV domain coloring of the complex plane. The center of this figure is the origin, and we show both positive and negative square roots of $B$ for visual clarity. In accordance with Figure~\ref{fig:LameFigure:a}, the plane is largely broken up into large cells that converge to the same point, and these cells approximately form a regular lattice. We can see that the algorithm itself breaks down for certain complex arguments as we move far from the origin; the boundaries between different basins of convergence grow noisier and more blurred in these regions. 

Finally, Figure~\ref{fig:HeunFigure} shows results for Heun equations with $\alpha=\beta$ and with $\gamma,\delta,\eps$ all at $1/3$, $2/3$, and $1$, respectively. Most eigenvalues appear to move only slowly away from their Lam\'e counterparts as parameters are changed. In particular, consistent with Conjecture~\ref{conj:conj2}, we appear to maintain a one-to-one correspondence between eigenvalues and the lattice $\frac{2\pi}{\Delta}\ovl{\Lambda}$. We emphasize that only the traces $t_{PQ}$ and $t_{QR}$ were used in identifying these values, and yet the third trace $t_{PR}$ always appears real within a $3\%$ relative error. As discussed above, these results provide strong evidence that Proposition~\ref{prop:general} and Corollaries~\ref{cor:general_darboux} and~\ref{cor:general_heun} can be strengthened to require only two pairwise traces to be real, as Beukers identified in the case of the Lam\'e equation (Proposition~\ref{prop:beukers}).


\begin{acknowledgments}
We would like to recognize the role of the MIT PRIMES-USA program for connecting the three authors and facilitating this research. In particular, we would like to recognize Pavel Etingof (MIT) for his insights and helpful discussions. The first author would like to acknowledge the support of an NDSEG Graduate Fellowship.
\end{acknowledgments}

\section*{Conflict of Interest Statement}
The authors have no conflicts to disclose.

\section*{Data Availability Statement}
We have published Python implementations of our algorithms on a GitHub repository, as well as instructions for their use: \url{https://github.com/ericc2023/HeunSimulation}.

\bibliography{sample}

@book{Fricke1897,
author = {Fricke, Robert and Klein, Felix},
location = {Leipzig},
publisher = {Teubner},
title = {Die gruppentheoretischen Grundlagen Bd. 1},
url = {http://eudml.org/doc/203007},
year = {1897},
}

@misc{reiter2011halphenstransformmiddleconvolution,
      title={{Halphen's} transform and middle convolution}, 
      author={Stefan Reiter},
      year={2011},
      eprint={0903.3654},
      archivePrefix={arXiv},
      url={https://arxiv.org/abs/0903.3654}, 
}

@book{cartan1995elementary,
  title={Elementary Theory of Analytic Functions of One or Several Complex Variables},
  author={Cartan, H.},
  isbn={9780486685434},
  lccn={lc95013507},
  series={Dover Books on Mathematics},
  year={1995},
  url={https://books.google.com/books?id=SlpwANmq5xQC},
  publisher={Dover Publications}
}

@book{etingof2011introduction,
  title={Introduction to Representation Theory},
  author={Etingof, P.I. and Golberg, O. and Hensel, S. and Liu, T. and Schwendner, A. and Vaintrob, D. and Yudovina, E.},
  isbn={9780821853511},
  lccn={2011004787},
  series={Student Mathematical Library},
  url={https://books.google.com/books?id=RS6IAwAAQBAJ},
  year={2011},
  publisher={American Mathematical Society}
}

@book{hall2013quantum,
  title={Quantum Theory for Mathematicians},
  author={Hall, B.C.},
  isbn={9781461471165},
  series={Graduate Texts in Mathematics},
  url={https://books.google.com/books?id=bYJDAAAAQBAJ},
  year={2013},
  publisher={Springer New York}
}

@book{simon2015comprehensive,
  title={A Comprehensive Course in Analysis},
  author={Simon, B.},
  isbn={9781470410988},
  series={A Comprehensive Course in Analysis},
  url={https://books.google.com/books?id=L6YJuAEACAAJ},
  year={2015},
  publisher={American Mathematical Society}
}

@article{Etingof:2023drx,
    author = "Etingof, Pavel and Frenkel, Edward and Kazhdan, David",
    title = "{A general framework and examples of the analytic {Langlands} correspondence}",
    eprint = "2311.03743",
    archivePrefix = "arXiv",
    doi = "10.4310/PAMQ.2024.v20.n1.a8",
    journal = "Pure Appl. Math. Quart.",
    volume = "20",
    number = "1",
    pages = "307--426",
    year = "2024"
}

@misc{gaiotto2011wallcrossinghitchinsystemswkb,
      title={Wall-crossing, {Hitchin} Systems, and the {WKB} Approximation}, 
      author={Davide Gaiotto and Gregory W. Moore and Andrew Neitzke},
      year={2011},
      eprint={0907.3987},
      archivePrefix={arXiv},
      url={https://arxiv.org/abs/0907.3987}, 
}

@article{beukers2007unitary,
  title={Unitary monodromy of {L}am{\'e} differential operators},
  author={Beukers, Frits},
  journal={Regular and Chaotic Dynamics},
  volume={12},
  number={6},
  pages={630--641},
  year={2007},
  publisher={Springer}
}

@article{sirota2006heun,
  title={The {Heun} equation and the {Darboux} transformation},
  author={Sirota, Yu N and Smirnov, Aleksandr Olegovich},
  journal={Mathematical Notes},
  volume={79},
  number={1},
  pages={244--253},
  year={2006},
  publisher={Springer}
}

@Article{heun1888theorie,
author={Heun, Karl},
title={Zur Theorie der Riemann'schen Functionen zweiter Ordnung mit vier Verzweigungspunkten},
journal={Mathematische Annalen},
year={1888},
month={Jun},
day={01},
volume={33},
number={2},
pages={161-179},
issn={1432-1807},
doi={10.1007/BF01443849},
url={https://doi.org/10.1007/BF01443849}
}

@article{beukers2002dwork,
  title={On {Dwork's} accessory parameter problem},
  author={Beukers, Frits},
  journal={Mathematische Zeitschrift},
  volume={241},
  number={2},
  pages={425--444},
  year={2002},
  publisher={Springer}
}

@article{keen1979moduli,
  title={Moduli of punctured tori and the accessory parameter of {Lam{\'e}’s} equation},
  author={Keen, L and Rauch, HE and Vasquez, AT},
  journal={Transactions of the American Mathematical Society},
  volume={255},
  pages={201--230},
  year={1979}
}

@article{nehari1949accessory,
  title={On the accessory parameters of a {Fuchsian} differential equation},
  author={Nehari, Zeev},
  journal={American Journal of Mathematics},
  volume={71},
  number={1},
  pages={24--39},
  year={1949},
  publisher={JSTOR}
}

@misc{Takemura2005HP,
      title={{Heun} equation and {Painlev\'e} equation}, 
      author={Kouichi Takemura},
      year={2005},
      eprint={math/0503288},
      archivePrefix={arXiv},
      url={https://arxiv.org/abs/math/0503288}, 
}

@Article{Eremenko2016rect,
author={Eremenko, Alexandre
and Gabrielov, Andrei},
title={Spherical Rectangles},
journal={Arnold Mathematical Journal},
year={2016},
month={Dec},
day={01},
volume={2},
number={4},
pages={463-486},
issn={2199-6806},
doi={10.1007/s40598-016-0055-5},
url={https://doi.org/10.1007/s40598-016-0055-5}
}

@article{2001math.....12179T,
author = {Takemura, Kouichi},
year = {2004},
month = {02},
pages = {},
title = {The {H}eun equation and the {C}alogero-{M}oser-{S}utherland system II: {P}erturbation and algebraic solution},
volume = {2004},
journal = {Electronic Journal of Differential Equations}
}

@article{10.1063/1.3367079,
    author = {Takemura, Kouichi},
    title = "{Integral transformation and Darboux transformation of Heun’s differential equation}",
    journal = {AIP Conference Proceedings},
    volume = {1212},
    number = {1},
    pages = {58-65},
    year = {2010},
    month = {03},
    issn = {0094-243X},
    doi = {10.1063/1.3367079},
    url = {https://doi.org/10.1063/1.3367079},
    eprint = {https://pubs.aip.org/aip/acp/article-pdf/1212/1/58/11745456/58\_1\_online.pdf},
}

@book{butcher2008numerical,
  title={Numerical Methods for Ordinary Differential Equations},
  author={Butcher, J.C.},
  isbn={9780470753750},
  url={https://books.google.com/books?id=opd2NkBmMxsC},
  year={2008},
  publisher={Wiley}
}

@inbook{alma9949400487002959,
author = {Pavel Etingof and Edward Frenkel and David Kazhdan},
address = {Providence, RI},
booktitle = {Integrability, Quantization, and Geometrym deducated to Boris Dubrovin, Vol. II},
isbn = {9781470455903},
lccn = {2020043148},
publisher = {American Mathematical Society},
series = {Proc.~Symp.~Pure Math., 103.2},
title = {An analytic version of the {Langlands} correspondence for complex curves},
year = {2021},
pages = {137--202},
url={https://arxiv.org/abs/1908.09677}
}

@Article{Goldman1988,
author={Goldman, William M.},
title={Topological components of spaces of representations},
journal={Inventiones mathematicae},
year={1988},
month={Oct},
day={01},
volume={93},
number={3},
pages={557-607},
issn={1432-1297},
doi={10.1007/BF01410200},
url={https://doi.org/10.1007/BF01410200}
}

@article{Morgan1984,
 ISSN = {0003486X, 19398980},
 URL = {http://www.jstor.org/stable/1971082},
 author = {John W. Morgan and Peter B. Shalen},
 journal = {Annals of Mathematics},
 number = {3},
 pages = {401--476},
 publisher = {[Annals of Mathematics, Trustees of Princeton University on Behalf of the Annals of Mathematics, Mathematics Department, Princeton University]},
 title = {Valuations, Trees, and Degenerations of Hyperbolic Structures, I},
 urldate = {2025-05-16},
 volume = {120},
 year = {1984}
}

@article{adachi_monodromy_2022,
	title = {Monodromy invariant {Hermitian} forms for second order {Fuchsian} differential equations with four singularities},
	volume = {42},
	issn = {1232-9274},
	url = {https://doi.org/10.7494/OpMath.2022.42.3.361},
	doi = {https://doi.org/10.7494/OpMath.2022.42.3.361},
	number = {3},
	journal = {Opuscula Math.},
	author = {Adachi, Shunya},
	year = {2022},
	pages = {361--391},
}

@article{adachi2023unitarymonodromiesrankfuchsian,
  title={Unitary Monodromies of Rank Two Fuchsian Systems with (<i>n</i> + 1) Singularities},
  author={Shunya Adachi},
  journal={Funkcialaj Ekvacioj},
  volume={67},
  number={3},
  pages={285-308},
  year={2024},
  doi={10.1619/fesi.67.285}
}

@misc{xia2021isomonodromy,
  doi = {10.48550/ARXIV.2101.02864},
  url = {https://arxiv.org/abs/2101.02864},
  author = {Xia, Jun and Xu, Shuai-Xia and Zhao, Yu-Qiu},
  
  title = {Isomonodromy sets of accessory parameters for {Heun} class equations},
  
  publisher = {arXiv},
      archivePrefix={arXiv},
eprint={2101.02864},
  
  year = {2021},
  
  copyright = {Creative Commons Attribution 4.0 International}
}

\end{document}